\documentclass[a4paper,11pt]{amsart}

\usepackage[utf8]{inputenc}		% font and encoding
\usepackage[T1]{fontenc}
\usepackage[english]{babel}

\usepackage{amsfonts}			% ams-packages
\usepackage{amsmath}
\usepackage{amssymb}
\usepackage{amsthm}

\usepackage{mathtools}

\usepackage[normalem]{ulem}

\usepackage{graphicx}			% images and drawing
\usepackage{tikz}
\usetikzlibrary{cd}
\usetikzlibrary{arrows.meta}

\usepackage{hyperref}			% hyperlinks
\hypersetup{colorlinks=false, urlcolor=black, linkcolor=black}
\usepackage{cleveref}

\usepackage{enumitem}			% arbitrary list items

\usepackage{color}				% colored text for signaling changes
\usepackage{float}

\usepackage{mathrsfs}

\newcommand{\N}{\mathbb{N}}						% Nonnegative integers
\newcommand{\Z}{\mathbb{Z}}						% Integers
\newcommand{\R}{\mathbb{R}}						% Reals
\newcommand{\C}{\mathbb{C}}						% Complex numbers

\renewcommand{\S}{\mathbb{S}}					% Sphere
\newcommand{\D}{\mathbb{D}}						% Disk
\newcommand{\B}{\mathbb{B}}
\newcommand{\A}{\mathbb{A}}
\newcommand{\eps}{\varepsilon}					% Epsilon shortcut
\newcommand{\dd}								% Differential d
{\mathop{}\!\mathrm{d}}						
\newcommand{\ddn}[1]							% Powers of a differential d
{\mathop{}\!\mathrm{d^{#1}}}

\newcommand{\abs}[1]							% Absolute value
{\left| #1 \right|}
\newcommand{\smallabs}[1]						% Small absolute value bars which won't scale to the argument.
{\lvert #1 \rvert}	
\newcommand{\norm}[1]							% Norm 
{\left\lVert #1 \right\rVert}	
\newcommand{\smallnorm}[1]						% Small norm bars which won't scale to the argument.
{\lVert #1 \rVert}						
\newcommand{\ip}[2]								% Inner product
{\left< #1 , #2 \right>}

\DeclareMathOperator{\intr}{int}				% Interior
\DeclareMathOperator{\vol}{vol}					% Volume (and the volume form)
\DeclareMathOperator{\spt}{spt}					% Support
\DeclareMathOperator{\diam}{diam}				
		
\DeclareMathOperator{\capac}{Cap}	
		
\DeclareMathOperator{\dist}{dist}

\DeclareMathOperator{\sgn}{sgn}

\newcommand{\loc}{\mathrm{loc}}

\renewcommand{\phi}{\varphi}

\def\Xint#1{\mathchoice
	{\XXint\displaystyle\textstyle{#1}}%
	{\XXint\textstyle\scriptstyle{#1}}%
	{\XXint\scriptstyle\scriptscriptstyle{#1}}%
	{\XXint\scriptscriptstyle\scriptscriptstyle{#1}}%
	\!\int}
\def\XXint#1#2#3{{\setbox0=\hbox{$#1{#2#3}{\int}$}
		\vcenter{\hbox{$#2#3$}}\kern-.5\wd0}}

\def\dashint{\Xint-}

\newtheorem{thm}{Theorem}[section]{\bf}{\it}
\newtheorem{lem}[thm]{Lemma}
\newtheorem{prop}[thm]{Proposition}
\newtheorem{cor}[thm]{Corollary}

{\bf}{\it}

\newtheorem{innercustomthm}{Theorem}[section]{\bf}{\it}
\newenvironment{customthm}[1]
{\renewcommand\theinnercustomthm{#1}\innercustomthm}
{\endinnercustomthm}

\newtheorem{innercustomprop}{Proposition}[section]{\bf}{\it}
\newenvironment{customprop}[1]
{\renewcommand\theinnercustomprop{#1}\innercustomprop}
{\endinnercustomprop}

{\bf}{\it}

\theoremstyle{definition}
\newtheorem{defn}[thm]{Definition}

\newtheorem{example}[thm]{Example}
\newtheorem{rem}[thm]{Remark}

\numberwithin{equation}{section}

\usepackage{amsthm}

\newtheorem*{customleminner}{\customlemname}

\begin{document}
\title[The Great Picard theorem for quasiregular values]{The Great Picard theorem for quasiregular values}

\author[I. Kangasniemi]{Ilmari Kangasniemi}
\address{Department of Mathematics and Statistics, 
P.O. Box 35 (MaD), FI-40014 University of Jyv\"askyl\"a, Finland.
}
\email{ilmari.k.kangasniemi@jyu.fi}

\author[Y. Zhu]{Yizhe Zhu}
\address{Department of Mathematics and Statistics, 
P.O. Box 35 (MaD), FI-40014 University of Jyv\"askyl\"a, Finland.
}
\email{yizhe.y.zhu@jyu.fi}

\subjclass[2020]{Primary 30C65; Secondary 35R45, 46E35.}
\date{\today}
\keywords{Quasiregular values, Great Picard theorem, Liouville theorem, Exterior domains}
\thanks{I.K.\ was supported by the Research Council of Finland grant \#369659. Y.Z.\ was supported by the Research Council of Finland grant \#334014 and by the Emil Aaltonen Foundation project ``\emph{Pohjois-Suomen kvasimaailman verkosto}''.}

\begin{abstract}
    We prove a version of Rickman's Great Picard theorem in the theory of quasiregular values. The version of the theorem that we prove is in fact stated in terms of omitted quasiregular values, unlike the previously known Little Picard theorem of the theory that is based on quasiregular values in the boundary of the image set. As a byproduct of the proof, we also prove a Riemann removability theorem in our setting, which ends up strengthening the conclusion of the previously resolved question of Astala, Iwaniec and Martin.
\end{abstract}

\maketitle

\section{Introduction}
Let $\Omega\subset\mathbb{R}^n$ be open and let $K\ge 1$ be a constant. We say that a map \( f:\Omega\to\R^n\) is {\em $K$-quasiregular} if $f$ is continuous, in the Sobolev space $W_{\text{loc}}^{1,n}(\Omega,\R^n)$,
and satisfies the distortion inequality
\begin{equation}\label{ineq: qr-distortion}
    |Df(x)|^n\le KJ_f(x)
\end{equation}
for almost every (a.e.) $x\in\Omega$. Here, $|Df(x)|$ is the operator norm of the weak derivative of $f$ at $x$, and $J_f$ denotes the Jacobian determinant of $f$. The planar version $n = 2$ of the theory of quasiregular mappings originates from the work of Gr\"otzsch ~\cite{grotzsch1928verzerrung} and Ahlfors~\cite{ahlfors1935theorie}. The $n$-dimensional counterpart of quasiregular mappings was introduced by Reshetnyak \cite{reshetnyak1966bounds,reshetnyak1967liouville,reshetnyak1968condition,reshetnyak1967space}, building upon the earlier study of $n$-dimensional quasiconformal maps in e.g.\ \cite{shabat1960theory, vaisala1961quasiconformal, gehring1962rings, zorich1962correspondence}; see also the founding works by Martio, Rickman, and V\"ais\"al\"a \cite{martio1969definitions,manrro1970distortion,martio1971topological}. 

The main draw of the theory of quasiregular maps is that it yields $n$-dimensional generalizations of many of the geometric results of holomorphic maps of a single complex variable. A notable example of this is Rickman's Picard theorem, which generalizes the Picard theorem of holomorphic maps. Rickman's result has both a Little Picard -version and a Great Picard -version, which we recall here; see~\cite{rickman1980picard} or \cite[Chapter IV]{rickman2012quasiregular} for details.

\begin{thm}[Rickman]
	Let $n \ge 2$ and $K \ge 1$. Then the following claims hold.
	\begin{enumerate}[label=(\roman*)]
		\item (Little Picard) There exists a positive integer $q = q(n, K)$ such that no entire $K$-quasiregular map $f \colon \R^n \to \R^n$ omits more than $q$ points, i.e.\ $\#(\R^n \setminus f(\R^n)) \le q$.
		\item (Great Picard) There exists a positive integer $q = q(n, K)$ such that for every $K$-quasiregular map $f \colon \R^n \setminus \overline{\B^n} \to \R^n$, if $f$ omits at least $q$ distinct points, then $\lim_{x \to \infty} f(x)$ exists in $\overline{\R^n} = \R^n \cup \{\infty\}$.
	\end{enumerate}
\end{thm}

Unlike the case $n = 2$ where $q(2, K) = 2$ matches the value for holomorphic maps, it is known that $\lim_{K \to \infty} q(n, K) = \infty$ when $n \ge 3$; see the counterexamples of Rickman \cite{rickman1985analogue}, Drasin, and Pankka \cite{drasin2015sharpness}. Several generalizations of the theorem have also been shown over the years, mostly focusing on the Little Picard -version. These include e.g.\ the version for closed manifold targets by Holopainen and Rickman ~\cite{holopainen1992picard,holopainen1998ricci}, the version for mappings of finite distortion satisfying suitable growth conditions by Rajala~\cite{rajala2004mappings}, and the analogue of Ahlfors' defect relations by Rickman, see \cite[Chapter V]{rickman1980picard}.

Recently, in~\cite{kangasniemi2024quasiregular,kangasniemi2025single}, Kangasniemi and Onninen introduced the theory of quasiregular values. In particular, given $y_0\in\R^n$, we say that a mapping $f\in W_{\text{loc}}^{1,n}(\Omega,\mathbb{R}^n)$ has a {\em $(K,\Sigma)$-quasiregular value at $y_0$} if it satisfies the inequality
\begin{equation}\label{defn: qr values into R^n}
	|Df(x)|^n\le KJ_f(x)+|f-y_0|^n\Sigma(x)
\end{equation}
for a.e. $x\in\Omega$, where $K\ge1$ and $\Sigma:\Omega\to[0,\infty)$ is measurable. The notion is motivated by generalized Beltrami equations arising in elliptic PDEs, and by the search for higher-dimensional counterparts to their planar theory; see e.g.\ \cite{astala2006calderon,astala2008elliptic,gilbarg1977elliptic}. The core idea behind the definition is that mappings with a quasiregular value satisfy single-value versions of many of the results of quasiregular maps. The theory has developed rapidly in the initial few years, see e.g.\ \cite{kangasniemi2022heterogeneous,kangasniemi2024correction,kangasniemi2025single,dolevzalova2024mappings,kangasniemi2024quasiregular,kangasniemi2024linear,heikkila2025quasiregular}.

A notable breakthrough moment for the theory was the proof that Rickman's Little Picard theorem has a counterpart for quasiregular values. This result, shown in \cite{kangasniemi2024quasiregular}, is as follows.

\begin{thm}[Kangasniemi-Onninen]\label{thm:Kangasniemi-Onninen_little_picard}
	 Let $K\ge 1$ and $\Sigma\in L^{1\pm\varepsilon}(\R^n)$ for some $\varepsilon>0$. Then no continuous $f\in W^{1,n}_{\loc}(\R^n,\R^n)$ has $q=q(n,K)\in\Z_{>0}$ distinct $(K,\Sigma)$-quasiregular values $y_1,\dots,y_q$ in the set $\partial f(\R^n)$.
\end{thm}

Crucially, this version of the Little Picard for quasiregular values is not formulated in terms of omitted quasiregular values, but instead concerns ones in $\partial f(\R^n)$; note that every such quasiregular value is omitted by the single-value Reshetnyak's theorem in \cite{kangasniemi2025single}. This concession is essential, as every compactly supported $f \in C^\infty_0(\R^n, \R^n)$ omits uncountably many $(1, \Sigma)$-quasiregular values with a bounded, compactly supported $\Sigma$. Despite this difference in statements, Rickman's Little Picard theorem is still a one-paragraph corollary of Theorem \ref{thm:Kangasniemi-Onninen_little_picard} for $\Sigma \equiv 0$, see~\cite[Remark 7.4]{kangasniemi2024quasiregular} 

In this paper, we show that Rickman's Great Picard Theorem also has a counterpart in the theory of quasiregular values. However, what to us is most striking about this result is that for the Great Picard -formulation of the theorem, \emph{the requirement that the omitted quasiregular values are in the boundary of the image of $f$ can be completely eliminated}: the result is purely about omitted quasiregular values, and the manner in which it implies the original Rickman's Great Picard theorem is legitimately obvious.

\begin{thm}\label{thm: classical great picard thm}
    Let $K\ge 1$ and $\Sigma\in L^{1\pm\varepsilon}(\R^n\setminus\overline{\B^n})$ for some $\varepsilon>0$. If a continuous map $f\in W^{1,n}_{\loc}(\R^n \setminus \overline{\B^n},\R^n)$ omits  $q=q(n,K)\in\Z_{>0}$ distinct $(K,\Sigma)$-quasiregular values $y_1,...,y_q\in \R^n\setminus f(\R^n\setminus\overline{\B^n})$, then the limit $\lim_{|x|\to\infty}
    f(x)\in\overline{\R^n}$ exists.
\end{thm}

The choice of the domain $\R^n\setminus\overline{\B^n}$ in Theorem \ref{thm: classical great picard thm} is merely for convenience. Indeed, by pre-composing $f$ with a conformal scaling, we observe that the same conclusion holds on $\R^n\setminus\overline{s\B^n}$ for every $s>0$, and hence on every punctured neighborhood of infinity.

The Little Picard theorem of quasiregular values, Theorem \ref{thm:Kangasniemi-Onninen_little_picard}, is also a short corollary of Theorem \ref{thm: classical great picard thm}; see Remark \ref{rem:little_picard_from_great_picard}. Coupled with the loss of the concession regarding the quasiregular values being in $\partial f(\R^n)$, it appears to us that the Great Picard is in some sense the more useful formulation of the Picard theorem for quasiregular values, as opposed to the Little Picard being more prominent in the theory of quasiregular maps.

We also provide a version of Theorem~\ref{thm: classical great picard thm} for punctured balls. Because of the integrability assumptions involved, the version at infinity is subtly stronger; the following theorem can be reduced to Theorem \ref{thm: classical great picard thm} with a conformal inversion, but the same is not true vice versa.

\begin{thm}\label{thm: classical great picard thm bounded}
    Let $K\ge 1$ and $\Sigma\in L^{1+\varepsilon}(\B^n)$ for some $\varepsilon>0$. If a continuous map $f\in W^{1,n}_{\loc}(\B^n \setminus \{0\},\R^n)$ omits $q=q(n,K)\in\Z_{>0}$ distinct $(K,\Sigma)$-quasiregular values $y_1,...,y_q\in \R^n\setminus f(\B^n \setminus \{0\})$, then the limit $\lim_{|x|\to0}
    f(x)\in\overline{\R^n}$ exists.
\end{thm}

As was the case for the Little Picard theorem in \cite{kangasniemi2024quasiregular}, our integrability assumptions on $\Sigma$ are sharp on the $L^p$-scale in both Theorems \ref{thm: classical great picard thm} and \ref{thm: classical great picard thm bounded}, thanks to essentially the same counterexamples; see the brief discussion at the end of Section \ref{sec: classical great picard}.

\subsection{The main ideas of the proof}

Our proof of Theorem \ref{thm: classical great picard thm} still follows the same strategy as the proof of the Little Picard theorem for quasiregular values by Kangasniemi and Onninen in \cite{kangasniemi2024quasiregular}, which in turn inherits a large portion of its structure from the proof of Bonk and Poggi-Corradini \cite{bonk2019rickman} for Rickman's Picard theorem. To our knowledge, this makes our proof the first time that a version of Rickman's Great Picard theorem is proven using the overall strategy of Bonk and Poggi-Corradini. In total, our more general case requires three key innovations or major adjustments, each corresponding to a different component of the proof.

\subsubsection{Hunting lemmas and the core of the proof} The first major difference is centered around the core of the proof that goes all the way back to \cite{bonk2019rickman}. This portion of the proof is encapsulated in the following proposition.

\begin{prop}\label{thm: finite q or finite Dh}
    Let $K\ge 1$ and $\Sigma\in L^{1+\varepsilon}_{\rm loc}(\R^n\setminus\overline{\B^n})\cap L^1(\R^n\setminus\overline{\B^n})$ for some $\varepsilon>0$. Then there exists a positive integer $q=q(n,K)\in\Z_{>0}$ with the following property: if a continuous mapping $h\in W^{1,n}_{\rm loc}(\R^n\setminus\overline{\B^n},\S^n)$ has a $(K,\Sigma)$-quasiregular value with respect to the spherical metric at $q$ distinct points $\omega_1,...\omega_q\in \S^n\setminus f(\R^n\setminus\overline{\B^n})$, all of which are asymptotic values of $f$ at infinity, then $\abs{Dh} \in{L^n(\R^n\setminus\overline{2\mathbb{B}^n})}$.
\end{prop}

The proof is based on a technique which is nowadays known as \emph{Rickman's hunting lemma}; the term has been coined by Bonk and Poggi-Corradini in \cite[Lemma 2.1]{bonk2019rickman}, though the ideas go back to Rickman's original works on the Picard theorem \cite{rickman1980picard}, and have since seen widespread use in the proofs of related obstruction results such as in \cite{Bonk-Heinonen_Acta}, \cite{Prywes_Annals}, and \cite{Heikkila-Pankka_Elliptic}. Thus, for our proof, we require a version of this lemma adapted for a Great Picard theorem, which selects the balls in a manner that they avoid the unit ball.

Rickman's original work on the Picard theorem in fact provides a version which aims to do just this; see \cite[Lemma 5.1]{rickman1980picard}. However, this stated version has a subtle flaw in its proof, and ends up being false; see the discussion and the provided counterexample in Example~\ref{exam: rickman} and Remark \ref{rem:Rickman's statement}. This flaw does not invalidate the currently known Great Picard theorem for quasiregular maps, since a later proof provided by Rickman in \cite[Theorem IV.2.27]{rickman2012quasiregular} does not refer to this hunting lemma. 

Instead, the appropriate version of the hunting lemma assumes an a-priori growth estimate for balls centered at 0 that is in a sense opposite to the doubling condition. For a precise statement, see Lemma \ref{lem: Rickman hunting lem}. The hunting lemma follows from similar ideas as are used in the modulus-based proof of \cite[Theorem IV.2.27]{rickman2012quasiregular}, as well as the proofs of various other similar results such as \cite[Proof of theorem 20]{Pankka-2006-GreatPicardBonkHeinonen}; thus, Lemma \ref{lem: Rickman hunting lem} is essentially stating these later ideas in the form of a hunting lemma. 

The proof of Proposition \ref{thm: finite q or finite Dh} then proceeds by executing the proof in \cite{kangasniemi2024quasiregular}, without the hunting lemma, for annuli $\A^n(r_0, r) = \B^n(0,r) \setminus \overline{\B^n}(0,r_0)$ centered at the origin. We need to re-prove the estimates in \cite{kangasniemi2024quasiregular} for annuli instead of balls, and the inner bound of the annulus causes various extra requirements and error terms in the estimates, but these problems all disappear when one performs the limiting process that naturally occurs at the end of the proof. The result of performing the argument in this manner is that for large enough $q$, one obtains the anti-doubling estimate required by the hunting lemma, and we can then follow the usual proof on the balls $B_j$ provided by the lemma. 

\subsubsection{Eliminating the finite spherical measure case}

Proposition \ref{thm: finite q or finite Dh} leaves unresolved the case where the image of the map has finite spherical volume, counting multiplicities. Resolving this step is a minor remark in the proof \cite{bonk2019rickman} for quasiregular mappings, but balloons into a significantly more complex ordeal in the proof \cite{kangasniemi2024quasiregular} for quasiregular values.

In our setting, the methods employed in \cite{kangasniemi2024quasiregular} lead into further significant technical challenges. While we believe these technical challenges are potentially surmountable, we elect to instead use a trick to avoid them altogether. Namely, we can reduce this step to the already resolved case where the domain of $f$ is all of $\R^n$, by sacrificing just a single one of the $q$ quasiregular values. We encapsulate this idea into the following proposition.

\begin{prop}\label{prop: extend to R^n}
    Let $K\ge 1$ and $\Sigma\in L^{1\pm\varepsilon}(\R^n\setminus\overline{\B^n})$ for some $\varepsilon>0$. Let $h\in W^{1,n}_{\rm loc} (\R^n\setminus\overline{\B^n},\S^n)$ be continuous with $\abs{Dh} \in L^n(\R^n\setminus\overline{2\B^n})$. Suppose that $h$ has a $(K,\Sigma)$-quasiregular value with respect to the spherical metric at $q$ distinct points $\omega_1,...\omega_q\in \S^n\setminus h(\R^n\setminus\overline{\B^n})$, all of which are asymptotic values of $h$ at infinity. 
    
    Then there exist a $\widetilde{\Sigma} \in L^{1\pm\varepsilon'}(\R^n)$ for some $\varepsilon'\in(0,\varepsilon)$, and a continuous $\tilde{h}\in W^{1,n}_{\rm loc} (\R^n,\S^n)$ such that $\smallabs{D\tilde{h}} \in L^n(\R^n)$ and $\tilde{h}$ has a $(K,\widetilde{\Sigma})$-quasiregular value with respect to the spherical metric at $q-1$ distinct points $\tilde{\omega}_1,...\tilde{\omega}_{q-1}\in \partial \tilde{h}(\R^n)$.
\end{prop}

The idea behind this step is that, under the assumptions, we can find an $r_0 > 2$ such that the image of $\partial\B^n(0,r_0)$ under $h$ is small. Then we can merely collapse $h$ to a constant value inside $\B^n(0,r_0)$, and with a small enough image, the maximum amount of $\omega_i$ this collapse adds to the image set is one.

\subsubsection{The Riemann removability theorem}

With the ideas outlined so far, we are able to prove a version of our main theorem for omitted asymptotic values. The last key step is hence to upgrade this into a result for omitted points which are not necessarily asymptotic values. To bridge this gap, we prove an analogue of the Riemann removability theorem in our setting. Recall that in classical complex analysis, the \emph{Riemann removability theorem} states that a bounded holomorphic mapping $f \colon \D \setminus \{0\} \to \C$ always has a holomorphic extension $f \colon \D \to \C$. 

In our setting, for a relatively general version of this result, we consider mappings satisfying a {\em distortion inequality with defect}
\begin{equation}\label{ineq: heterogeneous}
    |Df|^n\le KJ_f+\Sigma, 
\end{equation}
where $K\ge 1$ and $\Sigma\in L^{1+\varepsilon}_{\loc}$ with $\varepsilon>0$ and $\Sigma\ge 0$. Notably, for bounded mappings $f$, the definition \eqref{defn: qr values into R^n} of quasiregular values implies \eqref{ineq: heterogeneous} with an extra constant coefficient on $\Sigma$. For mappings satisfying \eqref{ineq: heterogeneous}, the Riemann removability theorem at infinity takes the following form. 

\begin{thm}\label{thm: riemann remo at infty}
    Suppose that $f\in W^{1,n}_\loc(\R^n\setminus \overline{\B^n},\R^n)$ satisfies the distortion inequality with defect \eqref{ineq: heterogeneous}
    with $K\ge 1$ and $\Sigma\in L^{1\pm\varepsilon}(\R^n\setminus \overline{\B^n})$ for some $\varepsilon>0$. If $f$ is bounded, then $f$ has a limit at infinity.
\end{thm}

The idea behind Theorem \ref{thm: riemann remo at infty} is to modify the chain of balls -argument used in the boundedness result \cite[Proposition 4.3]{kangasniemi2024quasiregular}. In particular, by setting up the chains of balls in a manner where they avoid the center of $\R^n$, the estimate proven therein can be upgraded to convergence at infinity.

We note that, as a corollary, Theorem \ref{thm: riemann remo at infty} immediately strengthens the main theorem of \cite{kangasniemi2022heterogeneous, kangasniemi2024correction}, i.e.\ the solution to the question of Astala, Iwaniec and Martin given in 
\cite[Section 8.5]{astala2008elliptic}. In particular, the assumption in the original theorem that $\lim_{x \to \infty} f(x) = 0$ can be weakened to just $f$ being bounded and $0$ being an asymptotic value of $f$ at infinity.

\begin{thm}\label{thm: liouville at omitted point}
    Let $K\ge 1$ and $\Sigma\in L^{1\pm\varepsilon}(\R^n)$ for some $\varepsilon>0$. Suppose that $f\in W^{1,n}_\loc(\R^n,\R^n)$ has a $(K,\Sigma)$-quasiregular value at $y_0\in \R^n$. If $f$ is bounded and $y_0$ is an asymptotic value of $f$ at infinity, then $f \equiv y_0$.
\end{thm}

\subsection{The structure of this paper}
The remainder of the paper is organized as follows. In Section~\ref{sec: preliminaries},
we fix notation and recall various preliminaries, such as the spherical formulation of quasiregular
values, the logarithmic potential, and the pull-back measures
associated with the spherical volume. 

In Section~\ref{sec: original tools}, we adapt the
Caccioppoli-type inequalities, pseudosupremum estimates, and criteria
for the existence of unbounded superlevel components shown in \cite{kangasniemi2024quasiregular} to our setting.
In Section \ref{sec: proof of the property}, still following the main outline of \cite{kangasniemi2024quasiregular}, we combine these estimates with the Hunting Lemma \ref{lem: Rickman hunting lem}
to prove Proposition~\ref{thm: finite q or finite Dh}. In Section~\ref{sec: picard type thms sec}, we treat the remaining case where the mapping has finite spherical volume, by proving and applying Proposition \ref{prop: extend to R^n}.

In Section~\ref{sec: riemann remo sec}, we then prove the Riemann removability theorem at infinity stated in Theorem \ref{thm: riemann remo at infty}, and obtain Theorem \ref{thm: liouville at omitted point} as a consequence. Finally, in Section~\ref{sec: classical great picard}, we assemble all these components into a proof of Theorems~\ref{thm: classical great picard thm} and \ref{thm: classical great picard thm bounded}.

\section{Preliminaries}\label{sec: preliminaries}
\subsection{Notation} A {\em domain} is a nonempty connected open subset of $\R^n$. For a set $E\subset\R^n$, $|E|$ denotes its Lebesgue measure, $\overline{E}$ its closure, and $\intr E$ its interior. If $h\in L^1(E)$ and $0<|E|<\infty$, then
\begin{equation*}
    (h)_E:=\dashint_E h=\frac{1}{|E|}\int_{E}h.
\end{equation*}

Suppose that $f$ is a real-valued function, we use $f^+$ and $f^-$ to denote its positive and negative parts, which are given by $f^+=\max (f,0)$ and $f^-=\max(0,-f)$. In particular, $f=f^{+}-f^-$. 

Throughout the manuscript, we use $C(a_1, a_2,\cdots,a_m)$ to denote a positive constant that depends on the parameters $a_i$. The value of $C(a_1, a_2,\cdots,a_m)$ may change in each
estimate even if the parameters remain the same. We also use the shorthand
$A\lesssim_{a_1, a_2,\cdots, a_m} B$ which stands for $A\le C(a_1, a_2,\cdots,a_m) B$. The shorthand $A\gtrsim_{a_1, a_2,\cdots, a_m} B$ is defined similarly. 

We then introduce some notation for a general metric space $(X,d_X)$. For $x,y\in X$, we denote the distance between \(x\) and \(y\) by $\dist_X(x,y)$. For a set $E\subset X$, we define
\[
    \diam_X(E):=\sup_{x,y\in E}\dist_X(x,y).
\]
We also write $B_X(a,r)$ for the open ball in $X$ centered at \(a\) with radius \(r\). 

When \(X=\R^n\) is equipped with the Euclidean metric, we omit the subscript \(X\) and write
\[
    \dist_{\R^n}(x,y)=|x-y|,\ \diam_{\R^n}(E)=\diam (E),\ \text{and}
    \ \B_{X}(a,r)=\B^n(a,r).
\]
We use the shorthand $\B^n = \B^n(0, 1)$ for the unit ball. Moreover, for $x_0\in\R^n$ and $0\le s<t<\infty$, we define the notation
\[
    \mathbb{A}^n(x_0;s,t) := \begin{cases}
        \mathbb{B}^n(x_0,t)\setminus\overline{\mathbb{B}^n}(x_0,s)
            & s > 0,\\
        \B^n(x_0, t) & s = 0
    \end{cases}
\]
for an annulus centered at $x_0$ with inner radius $s$ and outer radius $t$. Note that in the case $s = 0$ we elect to include the center point $x_0$ in the set. When $x_0=0$, we use the abbreviation $\A^n(0;s,t)=\A^n(s,t)$. We also adopt the standard convention that, if $B=\B^n(x,r)$ is a ball and  $Q = Q(a,r)=a+[-r/2,r/2]^n$ is a cube, then for $\lambda\in(0,\infty)$, we use $\lambda B$ to denote the ball $\B^n(x,\lambda r)$ and use $\lambda Q$ to denote the cube $Q(a,\lambda r)$.

If $U \subset \R^n$ is unbounded and $f \colon U \to Y$, where $Y$ is a metric space, we say that $y_0 \in Y$ is an \emph{asymptotic value} of $f$ at infinity if there exists a sequence of points $x_i\to\infty$ such that $f(x_i)\to y_0$. In particular, our definition of asymptotic values is based on sequences, unlike in some other contexts where they are defined using curves tending to $\infty$.

\subsection{Quasiregular values and maps into spheres}

We begin by recalling the theory of quasiregular values for mappings into $\S^n$, which was previously presented in \cite[Section 3]{kangasniemi2024quasiregular}. Let \(e_i\) denote the standard basis vectors of \(\R^n\), let
\(\langle\cdot,\cdot\rangle\) denote the Euclidean inner product on \(\R^n\),
and let \(|\cdot|\) denote the Euclidean norm. The \(n\)-dimensional unit sphere
\(\S^n\) consists of all 
\(w\in\R^{n+1}\) with \(|w|=1\). Recall that the inverse
\[
    s_n:\R^n\to \S^n\setminus\{-e_1\}
\]
of the stereographic projection is defined by
\[
    s_n(x)=\frac{1}{1+|x|^2}
    \bigl(1-|x|^2,\,2x_1,\,2x_2,\ldots,2x_n\bigr).
\]
The map \(s_n\) is then extended to \(\overline{\R^n}:=\R^n\cup\{\infty\}\) by setting
\(s_n(\infty)=-e_1\).

We recall that the spherical distance on \(\S^n\) is given by
\[
    \dist_{\S^n}(w_1,w_2)=
    \arccos\langle w_1,w_2\rangle,
    \qquad \omega_1,w_2\in\S^n,
\]
where we use the natural inclusion \(\S^n\subset\R^{n+1}\). We also define the spherical distance on \(\R^n\cup\{\infty\}\) by setting
\[
    \dist_{\overline{\R^n}}(x_1,x_2):=\dist_{\S^n}
    \bigl(s_n(x_1),s_n(x_2)\bigr),
    \qquad x_1,x_2\in\overline{\R^n}.
\]

We equip \(\S^n\) with the standard Riemannian metric induced by its inclusion \( \S^n\subset\R^{n+1}\), and with the standard orientation determined by the outward unit normal. We denote the corresponding volume form by
\(\vol_{\S^n}\). When $\S^n$ is equipped with this metric and
volume form, the map $s_n:\R^n\to\S^n\setminus{-e_1}$ is conformal. More precisely,
\begin{equation}\label{eq: Ds_n and its Jacobian}
    |Ds_n(x)|^n=J_{s_n}(x)=\frac{2^n}{\left(1+|x|^2\right)^n}.
\end{equation}
for every $x\in\R^n$. Moreover, given a set $E\subset\overline{\R^n}$, we denote its spherical measure
by $\vol_{\S^n}(E)$. By \eqref{eq: Ds_n and its Jacobian}, we see that
\begin{equation}\label{eq: volume of E on spheres}
    \vol_{\S^n}(E):=\int_E\frac{2^n\vol_n}{\left(1+|x|^2\right)^n}.
\end{equation}\

Let $\Omega\subset\R^n$ be a domain. Suppose that \( f\in W_{\mathrm{loc}}^{1,n}(\Omega,\R^n)
\). We then define a measurable map
\[
    h:\Omega\to\S^n,\quad \text{by}\quad h=s_n\circ f.
\]
Since \(s_n:\R^n\to\R^{n+1}\)
is a smooth Lipschitz map, the Lipschitz-Sobolev chain rule yields that \( h\in W_{\mathrm{loc}}^{1,n}(\Omega,\S^{n})\) and 
\(Dh(x)=Ds_n(f(x))Df(x)\)
for a.e.\ \(x\in\Omega\). In particular, the image of \(Dh(x)\) lies in $T_{h(x)}\S^n$
for a.e.\ \(x\in\Omega\), and hence 
\(Dh\) can be understood as a measurable map $T\Omega\to T\S^n$.
Consequently, we obtain a Jacobian
\(J_h\) of \(h\) by $J_h \vol_n =
h^*\vol_{\S^n}$. Since \(s_n\) is conformal, we obtain by 
\eqref{eq: Ds_n and its Jacobian} that
\begin{equation}\label{eq: Dh and J_h}
|Dh|^n=\frac{2^n|Df|^n}{\left(1+|f|^2\right)^n}\quad{\text{and}}\quad J_h=\frac{2^n J_f}{\left(1+|f|^2\right)^n}.
\end{equation}

Given $K\ge 1$, $y_0\in\S^n$, and
$\Sigma\in L_{\loc}^{1+\varepsilon}(\Omega)$ with $\varepsilon>0$
and $\Sigma\ge 0$, we say that a continuous mapping $h\in W^{1,n}_{\loc}(\Omega,\S^n)$ has a {\em $(K,\Sigma)$-quasiregular value with respect to the spherical metric at $w_0\in\S^n$} if $h$ satisfies
\begin{equation}\label{defn: qr values into S^n}
   |Dh(x)|^n\le KJ_h(x)+\dist^n_{\S^n}(h(x),w_0)\Sigma(x)
   \quad \text{for a.e.\ }x\in\Omega.
\end{equation}
We then compare the two definitions of quasiregular values given in \eqref{defn: qr values into R^n} and \eqref{defn: qr values into S^n}. The relationships between these two definitions are described in \cite[Lemma 3.1-3.2]{kangasniemi2024quasiregular}. For the reader's convenience, we restate them below.
 
\begin{lem}
    Let $\Omega\subset\R^n$ be a domain and let $f\in W_{\mathrm{loc}}^{1,n}(\Omega,\R^n)$. Let \(h=s_n\circ f \), let $w_0=s_n(y_0)$ for some $y_0\in\R^n$, let $K\ge 1$ and let 
    $\Sigma:\Omega\to[0,\infty)$ be measurable. Then the following statements hold. 
    \begin{enumerate}[label=(\roman*),ref=\thelem(\roman*),font=\normalfont]
    \item\label{lem: item comparison 1} 
    The following conditions are equivalent up to an extra constant factor $C=C(n,y_0)$ on $\Sigma$\textup{:} 
    \begin{enumerate}[label=(\alph*),font=\normalfont] 
    
    \item
    $f$ has a $(K,\Sigma)$-quasiregular value at $y_0$\textup{;}  
    
    \item $h$ has a $(K,\Sigma)$-quasiregular value with respect to the spherical metric at both $w_0$ and $s_n(\infty)$\textup{;}  
    
    \item $h$ satisfies
    \begin{equation*}
        \qquad\qquad|Dh|^n \le KJ_h +\dist^n_{\S^n}(h,s_n(y_0))\dist^n_{\S^n}(h,s_n(\infty))\Sigma,
    \end{equation*} 
    a.e. in $\Omega$.
    \end{enumerate} 
    \item\label{lem: item comparison 2}
    The following conditions are equivalent up to an extra constant factor $C=C(n,y_0)$ on $\Sigma$\textup{:} 
    \begin{enumerate}[label=(\alph*),font=\normalfont]
    \item $h$ has a $(K,\Sigma)$-quasiregular value with respect to the spherical metric at $w_0$\textup{;}  
    
    \item $f$ satisfies
    \[ 
        \qquad\qquad|Df|^n \le KJ_f +|f-y_0|^n(1+|f|^2)^{n/2}\Sigma, 
    \]
    a.e. in $\Omega$.
    \end{enumerate}
    \item\label{lem: item comparison 3}
    The following conditions are equivalent up to an extra constant factor $C=C(n)$ on $\Sigma$\textup{:} 
    \begin{enumerate}[label=(\alph*),font=\normalfont] 
    \item $h$ has a $(K,\Sigma)$-quasiregular value with respect to the spherical metric at $s_n(\infty)$\textup{;}  
    
    \item $f$ satisfies
    \[ 
        \qquad\qquad|Df|^n \le KJ_f +(1+|f|^2)^{n/2}\Sigma, 
    \] 
    a.e. in $\Omega$.
    \end{enumerate} 
    \end{enumerate}
\end{lem}

\subsection{The logarithmic potential}
We then recall the logarithmic potential function on $\S^n$, see e.g.\ \cite[Sect.~3]{bonk2019rickman} and \cite[Sect.~5.1]{kangasniemi2024quasiregular}. We first define a function \(S:[0,\infty)\to[0,1)\) by \begin{equation}\label{eq: defn of S} 
    S(r) = \frac{\vol_{\S^n}(\B^n(0,r))}{\vol_{\S^n}(\R^n)}.
\end{equation} 
By using \eqref{eq: volume of E on spheres}, one can see that 
\[ 
    S(r) = \frac{ 2^n\vol_{\S^{n-1}}(\R^{n-1})}{\vol_{\S^n}(\R^n)} \int_0^r \frac{t^{n-1}\,dt}{(1+t^2)^n}. 
\] 
In particular, \begin{equation}\label{eq: derivative of S} 
    S'(r) = \frac{C(n)r^{n-1}}{(1+r^2)^n}, 
\end{equation} 
and we obtain the following estimates describing the asymptotic behavior of $S(r)$ and $S'(r)$ for large and small $r$: 
\begin{equation}\label{ineq: estimates
for S and derivative of S}
    S(r)\lesssim_n \min(r^n,1), \qquad 
    S'(r)\lesssim_n \min(r^{n-1},r^{-(n+1)}). 
\end{equation} 
Next, we define a function \(H:[0,\infty)\to[0,\infty) \) by \begin{equation}\label{eq: defn of H}     H(r)=\int_0^r\frac{S^{\frac{1}
    {n-1}}(t)}{t}\,dt. 
\end{equation} 
Consequently, we have \begin{equation}\label{eq: derivative of H} 
    H'(r)=\frac{S^{\frac{1}{n-1}}(r)}{r},
\end{equation} and by applying
\eqref{ineq: estimates
for S and derivative of S},
we get the estimates \begin{equation}\label{ineq: estimates for H and derivative of h}     
    H(r)\lesssim_n\min
    \left( r^{\frac{n}{n-1}}, 1+|\log(r)|\right), 
    \qquad H'(r)\lesssim_n\min
    \left(r^{\frac{1}{n-1}}, r^{-1} \right).
\end{equation}
The logarithmic potential 
\(v_{\R^n}:\R^n\to[0,\infty)\) at infinity is then defined on $\R^n$ by \begin{equation}\label{eq: logarithmic potential} 
    v_{\R^n}(x)=H(|x|).
\end{equation} 
Since $v_{\R^n}$ is a real-valued radial function, we have \begin{equation}\label{eq: gradient of logarithmic potential} 
    \nabla v_{\R^n}(x)=H'(|x|)\frac{x}{|x|} \qquad\text{and}\qquad |\nabla v_{\R^n}(x)| = H'(|x|).
\end{equation}
Moreover, recall that if $v\in C^1(\R^n)$ is a function such that  
\(|\nabla v|^{n-2}\nabla v\in C^1(\R^n,\R^n)\), then the {\em $n$-Laplacian $\Delta_n v$} of $v$ is defined by 
\[ 
    \Delta_nv= \nabla\cdot\left(|\nabla v|^{n-2}\nabla v\right). 
\] 
We note that the $n$-Laplacian satisfies the identity
\[
    (\Delta_nv)\vol_n=d\left(|dv|^{n-2}\star dv\right). 
\] 
We record that the $n$-Laplacian of $v_{\R^n}$ is in fact exactly the density of the spherical volume; we refer to \cite[Lemma 3.1]{bonk2019rickman} for the proof. 
\begin{lem}\label{lem: logarithmic potential} 
    We have \(v_{\R^n}\in C^1(\R^n)\), \(|\nabla v_{\R^n}|^{n-2}\nabla v_{\R^n} \in C^1(\R^n,\R^n)\), and 
    \[ 
        \Delta_n v_{\R^n}(x)= \frac{2^n}{(1+|x|^2)^n}
        =J_{s_n}(x).
    \]
\end{lem}

\subsection{Pull-backs of the spherical volume}\label{subsec: defn of A_f}
We then recall the notation and main conclusions from \cite[Section 5.3]{kangasniemi2024quasiregular}. Let $K\ge 1$ and $\Sigma\in L^1(\Omega)$, and suppose that
$f\in W_{\loc}^{1,n}(\Omega,\R^n)$ is a continuous map for which
$s_n\circ f$ has a $(K,\Sigma)$-quasiregular value with respect to the
spherical metric at $s_n(\infty)$. Note that by
Lemma \ref{lem: item comparison 1}, this assumption is
true if $f$ has a $(K,\Sigma)$-quasiregular value at a point
$y_0\in\R^n$, up to an extra constant factor $C=C(n,y_0)$ on $\Sigma$. Moreover, the map $f$ satisfies
\begin{equation}\label{dis: Eucl interp of sphe qr values}
    |Df|^n
    \le KJ_f
    +C(n)(1+|f|^2)^{\frac n2}\Sigma
    \qquad\text{a.e. in }\Omega.
\end{equation}

We let $A_f^+$ and $A_f^-$ denote the positive and negative
parts of the pull-back of the spherical volume under $f$. That is,
if $E\subset\R^n$ is Lebesgue measurable, then
\begin{equation}\label{eq: positive negative spherical volume}
    A_f^+(E)
    =
    \int_E
    \frac{2^nJ_f^+}{(1+|f|^2)^n},
    \qquad
    A_f^-(E)
    =
    \int_E
    \frac{2^nJ_f^-}{(1+|f|^2)^n}.
\end{equation}
Both $A_f^+$ and $A_f^-$ are measures on $\R^n$. We then point out a fact about $A_f^-$ that was proven in \cite[Lemma 5.3]{kangasniemi2024quasiregular}; the original statement is for $\Omega = \R^n$, but the proof works on more general domains with no changes whatsoever.

\begin{lem}\label{lem: Negative part A_f^-}
    Let $f\in W_{\loc}^{1,n}(\Omega,\R^n)$ be a non-constant continuous
    map satisfying \eqref{dis: Eucl interp of sphe qr values}, where
    $K\ge1$, $\Sigma\ge0$, and $\Sigma\in L^1(\Omega)$. Then for every
    measurable set $E\subset\Omega$, we have
    \[
        A_f^-(E)\lesssim_n\Sigma(E)<\infty.
    \]
\end{lem}

In particular, for a map
$f\in W_{\loc}^{1,n}(\R^n,\R^n)\cap C(\R^n,\R^n)$ satisfying
\eqref{dis: Eucl interp of sphe qr values} with $K\ge1$ and
$\Sigma\in L^1(\R^n)$, the quantity \(A_f^+(E)-A_f^-(E)\)
is well-defined for every measurable set $E\subset\R^n$, since Lemma~\ref{lem: Negative part A_f^-} eliminates the possibility
of the expression evaluating to $\infty-\infty$. Thus, we may define
\[
    A_f(E)=A_f^+(E)-A_f^-(E)
\]
for every measurable set $E\subset\R^n$, obtaining a signed measure
$A_f$ on $\Omega$. The measure $A_f$ is the pull-back of the spherical
volume under $f$, and is given by
\begin{equation}\label{eq: pullback spherical volume}
    A_f(E)
    =
    \int_E
    \frac{2^nJ_f}{(1+|f|^2)^n}
\end{equation}
for every measurable set $E\subset\Omega$. We also use $|A_f|$ to denote the total variation measure of $A_f$,
which is naturally given by
\[
    |A_f|(E)
    =
    A_f^+(E)+A_f^-(E)
    =
    \int_E
    \frac{2^n|J_f|}{(1+|f|^2)^n}
\]
for every measurable set $E\subset\Omega$.

\subsection{Conformal capacity}

Let $\Omega \subset \R^n$ be open. Given two mutually disjoint compact sets \(E,F\subset\Omega,\) 
we define the {\em conformal capacity} of the condenser $(E,F)$ in $\Omega$ by 
\begin{equation}\label{defn: conformal capacity}
    \capac_\Omega(E,F)=\inf\left\{\int_{\Omega}|d\eta|^n: \eta\in C^\infty_0(\Omega),\eta|_E\ge1,\eta|_F\le 0\right\}.
\end{equation}
An equivalent definition is obtained if the assumption $\eta\in C^\infty_0(\Omega)$ in \eqref{defn: conformal capacity} is replaced by $\eta\in W^{1,n}_{0}(\Omega)\cap C(\Omega)$; see e.g.\ the discussion in \cite[pp.\ 27-28]{heinonen2018nonlinear}. We call a function $\eta\in W^{1,n}_{0}(\Omega)\cap C(\Omega)$ with $\eta|_E\ge1$ and $\eta|_F\le0$ {\em admissible} for the condenser $(E,F)$.

Next, we establish a variant of the conformal capacity estimate stated in \cite[Lemma 5.3]{bonk2019rickman}, stated for domains more general than $\R^n$. The proof is virtually identical with that of the original result.

\begin{lem}\label{lem: condenser capacity estimate lem}
    Let $q\ge 2$, and let $\Omega \subset \R^n$ be open and unbounded. For each $k\in\{1,...,q\}$, let $E_k$ and $F_k$ be relatively closed subsets of $\Omega$ such that $E_k\cap F_k=\emptyset$ for every $k$ and $F_l\cup F_k=\Omega$ whenever $l\neq k$. Let $B \subset \R^n$ be a ball of radius $r$, let $t>1$, and suppose that $A := t\overline{B} \setminus B$ is contained in $\Omega$. Suppose moreover that $B$ meets an unbounded component of $E_k$ for every $k\in\{1,...,q\}$. Let
    \begin{equation}\label{defn: E_{k,t} in the case of balls}
        E_{k,t}=E_k\cap(\overline{tB}\setminus B),\qquad F_{k,t}=F_k\cap(\overline{tB}\setminus B).
    \end{equation}
    Then
    \begin{equation}\label{ineq: upper bound for q n/(n-1) logt}
         q^{\frac{n}{n-1}}\log t\lesssim_n
        \sum_{k=1}^q\capac_\Omega(E_{k,t},F_{k,t}).
    \end{equation}
\end{lem}

\begin{proof}
    If $l\neq k$, we observe that since $F_l\cup F_k=\Omega$ and $E_l\cap F_l=\emptyset$, we have $E_l\subset\Omega\setminus F_l\subset F_k$. Consequently, $B$ also meets an unbounded component of $F_k$ for every $k\in\{1,...,q\}$. It follows that $\partial (sB)\cap E_k\neq\emptyset\neq\partial (sB)\cap F_k$ for every $s\in(1,t)$. Thus, the capacity estimate given in \cite[Lemma 3.3]{bonk2019rickman} for $\capac_{\R^n}$ applies, and
    \begin{multline*}
        \capac_\Omega(E_{k,t},F_{k,t}) \ge \capac_{\R^n}(E_{k,t},F_{k,t}) \\
        \gtrsim_n\int_{1}^{t}\frac{r \, ds}{[\mathcal{H}^{n-1}(\partial (sB)\setminus(E_k\cup F_k))]^{\frac{1}{n-1}}}.
    \end{multline*}
    %We note that the denominator $\mathcal{H}^{n-1}((\partial_s\mathbb{B})\setminus(E_k\cup F_k))$ in the above integral is non-zero for every $s\in(r,tr)$. Indeed, $(\partial_s\mathbb{B})\setminus(E_k\cup F_k)$ is an open subset of $\partial_s\mathbb{B}$, and $(\partial_s\mathbb{B})\setminus(E_k\cup F_k)$ is non-empty.

    Since the sets $\Omega\setminus(E_k\cup F_k)$ are pairwise disjoint, the sets $\partial (sB)\setminus(E_k\cup F_k)$ are pairwise disjoint for every $s\in(1,t)$. Thus, following the estimates in \cite[Lemma 6.2]{kangasniemi2024quasiregular}, we obtain that
    \begin{equation*}
        q^{\frac{n}{n-1}}\log t\lesssim_n
        \sum_{k=1}^q\capac_\Omega(E_{k,t},F_{k,t}).
    \end{equation*}
\end{proof}

\section{Caccioppoli-type inequalities}\label{sec: original tools}

So far, what we have recalled has been essentially unchanged from the proof of the little Picard theorem in \cite{kangasniemi2024quasiregular}. Here, however, we begin encountering slight deviations from the original proof. In particular, we establish the Caccioppoli-type inequalities used in the proof, following the presentation in \cite[Section~5]{kangasniemi2024quasiregular}. Since our mappings are defined on $\R^n\setminus\overline{\B^n}$ rather than on the whole space $\R^n$, the corresponding lemmas require modifications.

Notably, we need to prove the lemmas both for balls $\B^n(x, r) \subset \R^n\setminus\overline{2\B^n}$, where $x \in \R^n\setminus\overline{2\B^n}$ and $r > 0$, and for annuli of the form $\A^n(2, 2R)$ where $R > 1$. We condense these into a single case by proving the estimates for annuli $\A^n(x_0;\lambda r,R)$ with $\A^n(x_0;\lambda r,2R) \subset \R^n\setminus\overline{2\B^n}$ for some $\lambda \in (0,1)$, where we notably allow for $r = 0$. While the proofs follow the same outline as before, the estimates will have to also account for the inner portion of the annulus, which necessitates some changes to the statements and arguments.

\subsection{Vanishing Jacobian on bounded super-level components}
We begin with the following general result, which is analogous to \cite[Lemma~5.2]{kangasniemi2024quasiregular} with a mostly unchanged proof. Since the mapping $f$ is defined on $\R^n\setminus\overline{\B^n}$, the inner boundary
$\partial\B^n$ has to be taken into account. In particular, when considering bounded components of the relevant super-level sets, we
restrict our attention to those that are compactly contained in
$\R^n\setminus\overline{\B^n}$.
\begin{lem}\label{lem: zero Jacobian}
    Suppose that $f\in W^{1,n}_\loc(\R^n\setminus
    \overline{\mathbb{B}^n},\R^n)$ is continuous. Let $y_0\in\R^n$, let $r>0$, and let $U$ be a bounded component of $f^{-1}(\R^n\setminus \overline{\mathbb{B}^n}(y_0,r))$ such that $U$ is compactly contained in $\R^n\setminus
    \overline{\B^n}$. Then for any continuous function $\Phi:[r,\infty)\to[0,\infty)$, we have 
    \begin{equation*}
        \int_U\Phi(|f-y_0|)J_f=0.
    \end{equation*}
\end{lem}
\begin{proof}
    Since $U$ is compactly contained in $\R^n\setminus
    \overline{\B^n}$ and since $f$ is continuous, $f(\overline{U})$ is compact, and hence there exists $y_1\in \R^n\setminus(\overline{\mathbb{B}^n}(y_0,r)\cup f(\overline{U}))$. By also using the continuity of $\Phi$, it follows that $\Phi(|f-y_0|)J_f$ is integrable over $U$. By a Sobolev change of variables, see e.g.\ \cite[Theorem 5.27]{fonseca1995degree}, we have 
    \begin{equation*}
        \int_U\Phi(|f-y_0|)J_f=\int_{\R^n\setminus\overline{\mathbb{B}^n}(y_0,r)}\Phi(|y-y_0|)\deg(f,y,U)\vol_n(y).
    \end{equation*}
    Since $U$ is a connected component of $f^{-1}(\R^n\setminus\overline{\mathbb{B}^n}(y_0,r))$ with $\partial U \subset \R^n\setminus\overline{\B^n}$, we have $f(\partial U)\subset\partial \mathbb{B}^n(y_0,r)$. Thus, since $\R^n\setminus\overline{\mathbb{B}^n}(y_0,r)$ is connected, we have $\deg(f,y,U)=\deg(f,y_1,U)=0$ for every $y\in\R^n\setminus\overline{\mathbb{B}^n}(y_0,r)$, see for instance \cite[Theorem 2.1 and Theorem 2.3 (3)]{fonseca1995degree}. In conclusion,
    \begin{equation*}
        \int_U\Phi(|f-y_0|)J_f=0.
    \end{equation*}
\end{proof}

\subsection{Measure estimates and Caccioppoli-type inequalities.}
In section \ref{subsec: defn of A_f}, we introduced the measure $A_f$. We now turn to deriving several estimates for $A_f$. We begin with a counterpart to \cite[Lemma~5.4]{kangasniemi2024quasiregular}, a technical Caccioppoli-type estimate that will be used repeatedly in the subsequent arguments.
\begin{lem}\label{lem: Caccioppoli type ineq}
    Let $f\in W^{1,n}_\loc(\R^n\setminus
    \overline{\B^n},\R^n)$ be a non-constant continuous mapping that satisfies \eqref{dis: Eucl interp of sphe qr values}, where $K\ge1$, $\Sigma\ge0$, and $\Sigma\in L^{1}(\R^n\setminus
    \overline{\B^n})\cap L^{1+\varepsilon}_\loc(\R^n\setminus
    \overline{\B^n})$ for some $\varepsilon>0$. Then for every $L\ge0$ and $\eta\in C^\infty_0(\R^n\setminus
    \overline{\B^n})$ with $\eta\ge 0$, we have 
    \begin{multline*}
        \int_{\{v_{\R^n}\circ f<L\}}\eta^n\left(|dv_{\R^n}|^n\circ f\right)|Df|^n\lesssim_n K^nL^n\int_{\{v_{\R^n}\circ f<L\}}|d\eta|^n\\
        +KL\int_{\{v_{\R^n}\circ f<L\}}\eta^n {\rm d}|A_f|+\int_{\{v_{\R^n}\circ f<L\}}\eta^n\Sigma.
    \end{multline*}
\end{lem}

\begin{proof}
    The proof is exactly the same as that of \cite[Lemma~5.4]{kangasniemi2024quasiregular}; the original result is stated for $f \colon \R^n \to \R^n$, but the argument does not involve any consideration of the values of $f$ outside $\spt \eta$, and hence applies on more general domains.  
\end{proof}

As an immediate consequence of Lemma~\ref{lem: Caccioppoli type ineq}, suitable choices of cut-off functions yield the following estimate on annuli.

\begin{cor}\label{cor: truncate estimate of |du|^n on u<L}
    Let $f\in W^{1,n}_\loc(\R^n\setminus
    \overline{\B^n},\R^n)$ be a non-constant continuous mapping that satisfies \eqref{dis: Eucl interp of sphe qr values}, where $K\ge1$, $\Sigma\ge0$, and $\Sigma\in L^{1}(\R^n\setminus
    \overline{\B^n})\cap L^{1+\varepsilon}_\loc(\R^n\setminus
    \overline{\B^n})$ for some $\varepsilon>0$. Let $u=v_{\R^n}\circ f$. For $x_0\in\R^n$, $\lambda\in(0,1)$, and $0\le r<R$, set $A:=\A^n(x_0;r,R)$ and $A_{\lambda}:=\A^n(x_0;\lambda r,2R)$. If $A_\lambda\subset\R^n\setminus\overline{\B^n}$, then for every $L>0$, we have 
    \begin{equation*}
        \int_{A\cap\{u<L\}}\abs{du}^n\lesssim_{n,\lambda} K^nL^n+ KL|A_f|(A_t)+\Sigma(\R^n\setminus\overline{\B^n}) 
    \end{equation*}
\end{cor}

\begin{proof}
     We fix a cutoff function $\eta\in C^\infty_0(A_\lambda,[0,1])$ satisfying $\eta\equiv 1$ on $A$, with $\|d\eta\|_{L^\infty}\le 2R^{-1}$ on $\A^n(x_0;R,2R)$ and $\|d\eta\|_{L^\infty}\le 2/r(1-\lambda)$ on $\A^n(x_0;\lambda r,r)$. Note that this remains possible in the special case $r = 0$ where $\A^n(x_0;\lambda r,r) = \emptyset$, since $x_0 \in A$ in that case.
     
     Since $v$ is $C^1$, the chain rule of $C^1$ and Sobolev functions yields that
    \begin{equation*}
        \abs{du}=\abs{df^*v}=\abs{f^*dv}\le(\abs{dv}\circ f)\abs{Df}^n.
    \end{equation*}
    Hence,
    \begin{equation*}
        \int_{A\cap\{u<L\}}\abs{du}^n\le\int_{\{u<L\}}\eta^n\left(\abs{dv}^n\circ f\right)\abs{Df}^n.
    \end{equation*}
    By using Lemma \ref{lem: Caccioppoli type ineq}, we obtain that
    \begin{align*}
        &\int_{\{u<L\}}\eta^n\left(\abs{dv}^n\circ f\right)\abs{Df}^n\\
        &\qquad \begin{multlined} \lesssim_n K^nL^n
        \int_{A_\lambda\cap\{u<L\}}\abs{d\eta}^n
        +KL\int_{A_\lambda\cap\{u<L\}}\eta^n\, 
            {\rm d}|A_f|
        +\int_{A_\lambda\cap\{u<L\}}\eta^n\Sigma \end{multlined}\\
        &\qquad\le K^nL^n\|d\eta\|^n_{L^n}+KL|A_f|(A_\lambda)+\Sigma(\R^n\setminus\overline{\B^n}).
    \end{align*}
    Since
    \begin{multline*}
        \|d\eta\|^n_{L^n}
        \le \|d\eta\|^n_{L^\infty}\vol_n(\spt d\eta)\\
        \lesssim_n \frac{|\A^n(x_0;tr,r)|}{r^n(1-\lambda)^n}+\frac{|\A^n(x_0;R,2R)|}{R^n}\lesssim_n \frac{1}{(1-\lambda)^n}+1,
    \end{multline*}
    the claim follows.
\end{proof}
\subsection{Pseudosupremum estimates}
Next, following \cite[Definition~5.6]{kangasniemi2024quasiregular}, we define the pseudosupremum of a continuous function $u:\R^n\setminus \overline{\B^n}\to[0,\infty)$ as follows.
\begin{defn}
    Let $u:\R^n\setminus \overline{\B^n}\to[0,\infty)$ be continuous, and let $E\subset\R^n\setminus \overline{\B^n}$. The {\em pseudosupremum} $\widetilde{\sup}_{E} u$ of $u$ over $E$ is defined by
    \[
        \widetilde{\sup}_{E} u = \sup \{t \in \R : E
        \text{ meets an unbounded component of } u^{-1}(t, \infty)\}. 
    \]
\end{defn}
We then establish a counterpart of
\cite[Lemma~5.8]{kangasniemi2024quasiregular} for the domain $\R^n\setminus \overline{\B^n}$. In addition to the pseudosupremum condition, we require the
level $L$ to lie above the maximum of $u:=v_{\R^n}\circ f$ on
$\partial 2\B^n$. This additional condition ensures that the bounded
components of $\{u>L\}$ meeting the sets under consideration are
compactly contained in
$\R^n\setminus\overline{\B^n}$, so that Lemma~\ref{lem: zero Jacobian} can be applied.

\begin{lem}
    \label{lem: truncate u ge L, annulus}
    Let $f\in W^{1,n}_\loc(\R^n\setminus
    \overline{\B^n},\R^n)$ be a non-constant continuous mapping that satisfies \eqref{dis: Eucl interp of sphe qr values}, where $K\ge1$, $\Sigma\ge0$, and $\Sigma\in L^{1}(\R^n\setminus
    \overline{\B^n})\cap L^{1+\varepsilon}_\loc(\R^n\setminus
    \overline{\B^n})$ for some $\varepsilon>0$. Let $u =v_{\R^n}\circ f$. Then for every bounded open domain $D \subset \R^n \setminus \overline{2\B^n}$, every $\eta\in C^\infty_0(D)$ with $\eta\ge 0$, and every 
    \begin{equation*}
        L>\max\left(\widetilde{\sup}_{D}u, \max_{\partial 2\B^n}u\right),
    \end{equation*}
    we have
    \begin{equation*}
         \quad\int_{\{u\ge L\}}\eta^n\left(\abs{dv}^n\circ f\right)\abs{Df}^n\lesssim_n\|\eta\|_{L^\infty}\Sigma(\mathbb{R}^n\setminus\overline{\B^n}).
    \end{equation*}
\end{lem}

\begin{proof}
    Let $U=\{u>L\}$. By definition, $D$ only meets bounded components of $U$, all of which are contained in $\R^n\setminus\overline{2\mathbb{B}^n}$ since $L> \max_{\partial 2\B^n}u$. We denote the union of these components of $U$ that meet $D$ by $U_D$. Now, recalling that $\spt\eta\subset D$, and that $f$ satisfies \eqref{dis: Eucl interp of sphe qr values}, we may use Lemma \ref{lem: zero Jacobian} with $y_0=0$ and $\Phi(|f-y_0|)=\Phi(|f|)=\left(\abs{dv}^n\circ f\right)$ to obtain the estimate
    \begin{multline*}
        \int_{\{u>L\}}\eta^n\left(\abs{dv}^n\circ f\right)\abs{Df}^n\le\|\eta\|_{L^\infty}\int_{U_D}\left(\abs{dv}^n\circ f\right)\abs{Df}^n\\
        \lesssim_n\|\eta\|_{L^\infty}\int_{U_D}\left(\abs{dv}^n\circ f\right)\left(1+|f|^2\right)^\frac{n}{2}\Sigma.
    \end{multline*}
    Since $u=H(|f|)$ and $H$ is a strictly increasing function, we obtain that $\{u=L\}=\{|f|=C_L\}$ for some constant $C_L$ depending on $L$. Hence, $f(\{u=L\})\subset\partial\mathbb{B}^n(0,C_L)$, and 
    so the image of $\{u=L\}$ under $f$ has zero Hausdorff $n$-measure. Thus, $J_f=0$ a.e. in $\{u=L\}$. By \eqref{dis: Eucl interp of sphe qr values}, we also have $\abs{Df}^n\lesssim_n(1+\abs{f}^2)^{n/2}\Sigma$ a.e. in $\{u=L\}$. Hence, the previous estimate can be improved to
    \begin{multline*}
        \int_{\{u\ge L\}}\eta^n\left(\abs{dv}^n\circ f\right)\abs{Df}^n\\
        \lesssim_n\|\eta\|_{L^\infty}\int_{U_D\cup \{u=L\}}\left(\abs{dv}^n\circ f\right)\left(1+|f|^2\right)^\frac{n}{2}\Sigma.
    \end{multline*}
    Since $\left(\abs{dv}\circ f\right)\left(1+|f|^2\right)^\frac{1}{2}\le C(n)$ by \eqref{ineq: estimates for H and derivative of h} and \eqref{eq: gradient of logarithmic potential}, we obtain the desired estimate
    \begin{multline*}
        \int_{\{u\ge L\}}\eta^n\left(\abs{dv}^n\circ f\right)\abs{Df}^n\\
        \lesssim_n\|\eta\|_{L^\infty}\Sigma(U_D\cup\{u=L\})\le\|\eta\|_{L^\infty}\Sigma(\mathbb{R}^n\setminus\overline{\B^n}).
    \end{multline*}
\end{proof}

The preceding lemma provides control of the energy on the superlevel set
$\{u\ge L\}$ whenever $L$ lies above both the relevant pseudosupremum and the boundary level $\max_{\partial 2\B^n}u$. Combining this estimate and the argument of \cite[Lemma~5.9]{kangasniemi2024quasiregular}, we obtain the estimates for the measure $|A_f|$ on annuli and balls.

\begin{lem}\label{lem: estimate of |A_f|}
    Let $f\in W^{1,n}_\loc(\R^n\setminus
    \overline{\B^n},\R^n)$ be a non-constant continuous mapping that satisfies \eqref{dis: Eucl interp of sphe qr values}, where $K\ge1$, $\Sigma\ge0$, and $\Sigma\in L^{1}(\R^n\setminus
    \overline{\B^n})$. Let $u=v_{\R^n}\circ f$. For $x_0\in\R^n$, $\lambda\in(0,1)$, and $0\le r<R$, set $A:=\A^n(x_0;r,R)$ and $A_{\lambda}:=\A^n(x_0;\lambda r,2R)$. If $A_\lambda\subset\R^n\setminus\overline{2\B^n}$, then for 
    \begin{equation*}
        L_0:=\max\left(\widetilde{\sup}_{A_\lambda}u, \max_{\partial 2\B^n}u\right),
    \end{equation*}
    we have 
    \begin{equation*}
        |A_f|(A)\lesssim_{n,\lambda} K^{n-1}L_0^{n-1}+C(n,\Sigma(\R^n\setminus
    \overline{\B^n})).
    \end{equation*}
\end{lem}

\begin{proof}
    Fix a cutoff function $\eta\in C^\infty_0 (A_\lambda,[0,1])$ satisfying $\eta\equiv 1$ on $A$, with $\|d\eta\|_{L^\infty}\le 2R^{-1}$ on $\A^n(x_0;R,2R)$ and $\|d\eta\|_{L^\infty}\le 2/r(1-\lambda)$ on $\A^n(x_0;\lambda r,r)$. Note again that this remains possible in the case $r = 0$ with $\A^n(x_0;\lambda r,r) = \emptyset$.

    We estimate that 
    \begin{equation*}
        \abs{A_f}(A)\le\int_{\R^n\setminus\overline{2\mathbb{B}^n}}\eta^n \,{\rm d}|A_f|\le\int_{\R^n\setminus\overline{2\mathbb{B}^n}}\eta^n \,{\rm d}A_f+A_f^-(A_\lambda).
    \end{equation*}
    By Lemma \ref{lem: Negative part A_f^-}, we have $A_f^-(A_\lambda)\lesssim_n\Sigma(A_\lambda)$,
    and by Lemma \ref{lem: logarithmic potential}, we have
    \begin{equation*}
        d(|dv|^{n-2}\star dv)=(\Delta_nv)\vol_n=J_{s_n}\vol_n=s_n^*\vol_{\mathbb{S}^n}.
    \end{equation*} 
    Hence, we obtain
    \begin{multline*}
        \int_{\R^n\setminus\overline{2\mathbb{B}^n}}\eta^n \,{\rm d}A_f=\int_{\R^n\setminus\overline{2\mathbb{B}^n}}\eta^nf^*s_n^*\vol_{\mathbb{S}^n}\\=\int_{\R^n\setminus\overline{2\mathbb{B}^n}}\eta^nf^*d(|dv|^{n-2}\star dv)=\int_{\R^n\setminus\overline{2\mathbb{B}^n}}\eta^ndf^*(|dv|^{n-2}\star dv)\\
        \le n\int_{\R^n\setminus\overline{2\mathbb{B}^n}}\eta^{n-1}|d\eta||f^*(|dv|^{n-2}\star dv)|,
    \end{multline*}
    where we used the fact that $d$ and $f^*$ commute since the form $|dv|^{n-2}\star dv$ is $C^1$-smooth, and the fact that $\eta\in C^\infty_0(A_\lambda,[0,1])$ which eliminates the boundary term when integrating by parts. Furthermore, by using
    \begin{equation*}
        |f^*(|dv|^{n-2}\star dv)|\le |Df|^{n-1}(|dv|^{n-1}\circ f)
    \end{equation*}
    and H\"older's inequality, we obtain that
    \begin{multline*}
        \int_{\R^n\setminus\overline{2\mathbb{B}^n}}\eta^{n-1}|d\eta||f^*(|dv|^{n-2}\star dv)|\\
        \le \left(\int_{\R^n\setminus\overline{2\mathbb{B}^n}}|d\eta|^n\right)^{\frac{1}{n}}\left(\int_{\R^n\setminus\overline{2\mathbb{B}^n}}\eta^n\left(|dv|^n\circ f\right)|Df|^n\right)^{\frac{n-1}{n}}.
    \end{multline*}
    By our assumptions on $\eta$, we have 
    \[
        \|d\eta\|^n_{L^n(\R^n\setminus
        \overline{2\mathbb{B}^n})}
        \lesssim_n 1+\frac{1}{(1-\lambda)^n}.
    \]
    In conclusion,
    \begin{multline}\label{ineq: estimate for |A_f|}
        |A_f(A)|\le\int_{\R^n\setminus\overline{2\mathbb{B}^n}}\eta^n \,{\rm d}|A_f|\\
        \lesssim_{n,\lambda}\left(\int_{\R^n\setminus\overline{2\mathbb{B}^n}}\eta^n\left(|dv|^n\circ f\right)|Df|^n\right)^{\frac{n-1}{n}}+\Sigma(\R^n\setminus\overline{\B^n}).
    \end{multline}
    
    We then estimate the integral in \eqref{ineq: estimate for |A_f|}. Let $L>L_0$. By Lemma \ref{lem: truncate u ge L, annulus} with $D = A_\lambda$, we obtain
    \begin{equation*}
        \int_{\{u\ge L\}}\eta^n\left(|dv|^n\circ f\right)|Df|^n\lesssim_{n,\lambda}\Sigma(\mathbb{R}^n\setminus\overline{\B^n}).
    \end{equation*}
    In the remaining set $\{x\in A_\lambda: u< L\}$, we use Lemma \ref{lem: Caccioppoli type ineq} together with the fact that $\|d\eta\|_{L^n}\lesssim_{n,\lambda} 1$, which yields the estimate
        \begin{multline}\label{ineq: ineq: estimate for |A_f| step 2}
        \int_{\{u<L\}}\eta^n\left(|dv|^n\circ f\right)|Df|^n\\
        \lesssim_{n,\lambda} K^nL^n
        +KL\int_{\{u<L\}}\eta^n {\rm d}|A_f|+\Sigma(\R^n\setminus\overline{\B^n}).
    \end{multline}
    Next, chaining together \eqref{ineq: estimate for |A_f|} and \eqref{ineq: ineq: estimate for |A_f| step 2} and using the elementary inequality $(a+b)^p\lesssim_p a^p+b^p$ for $a$, $b$, $p\ge0$, we obtain 
    \begin{multline}\label{ineq: estimate for |A_f| step 3}
        \int_{\R^n\setminus\overline{\mathbb{B}^n}}\eta^n \,{\rm d}|A_f|\lesssim_{n,\lambda} \Sigma(\R^n\setminus\overline{\B^n})+[\Sigma(\R^n\setminus\overline{2\B^n})]^{\frac{n-1}{n}}\\
        +(KL)^{n-1}
        +(KL)^{\frac{n-1}{n}}\left(\int_{\R^n\setminus\overline{2\mathbb{B}^n}}\eta^n \,{\rm d}|A_f|\right)^{\frac{n-1}{n}}.
    \end{multline}
    Let \(C_2=C_2(n,\lambda)\) denote the implicit constant in
    \eqref{ineq: estimate for |A_f| step 3}. By Young's inequality, for every
    \(\delta>0\),
    \begin{multline*}
        (KL)^{\frac{n-1}{n}}\left(\int_{\R^n\setminus\overline{2\mathbb{B}^n}}\eta^n \,{\rm d}|A_f|\right)^{\frac{n-1}{n}}\\
        \le C(n,\lambda,\delta) (LK)^{n-1}+\delta \int_{\R^n\setminus\overline{2\mathbb{B}^n}}\eta^n \,{\rm d}|A_f|.
    \end{multline*}
    Choosing $\delta=(2C_2)^{-1}$, and noting that the second term on the right-hand side is finite, we may absorb it into the left-hand side of \eqref{ineq: estimate for |A_f| step 3}. Hence,
    \begin{equation*}
        |A_f|(A)\lesssim_{n,\lambda} (KL)^{n-1}+\Sigma(\R^n\setminus\overline{\B^n})+[\Sigma(\R^n\setminus\overline{\B^n})]^{\frac{n-1}{n}}.
    \end{equation*}
    Since $L>L_0$ is arbitrary, the claim follows.
\end{proof}

\subsection{Existence of unbounded components}

We conclude this section with a counterpart of \cite[Lemma~5.10]{kangasniemi2024quasiregular}. In our setting, we show that a sufficiently large $|A_f|$-measure on an annulus forces a corresponding
larger region to meet an unbounded component of a prescribed superlevel
set of $u$.

\begin{lem}\label{lem: unbounded components}
    Let $f\in W^{1,n}_\loc(\R^n\setminus
    \overline{\B^n},\R^n)$ be a non-constant continuous mapping that satisfies \eqref{dis: Eucl interp of sphe qr values}, where $K\ge1$, $\Sigma\ge0$, and $\Sigma\in L^{1}(\R^n\setminus
    \overline{\B^n})\cap L^{1+\varepsilon}_\loc(\R^n\setminus
    \overline{\B^n})$ for some $\varepsilon>0$.  Let $u=v_{\R^n}\circ f$. For $x_0\in\R^n$, $\lambda\in(0,1)$, and $0\le r<R$, set $A:=\A^n(x_0;r,R)$ and $A_{\lambda}:=\A^n(x_0;\lambda r,2R)$. Suppose that $A_\lambda\subset\R^n\setminus\overline{2\B^n}$. Then for every
    \[
        t > \max_{\partial 2\B^n}u,
    \]
    there exists $s=s(n,K,\Sigma, t)>0$ such that $A_\lambda$ meets an unbounded component of 
    $\{u>t\}$ contained in $\R^n\setminus{\overline{2\mathbb{B}^n}}$ whenever $|A_f|(A)>s$
\end{lem}
\begin{proof}
     Let $A,A_\lambda$ be annuli as in the statement, and let $t>0$ be such that $A_\lambda$ meets no unbounded component of $\{u>t\}$ that is contained in $\R^n\setminus{\overline{2\mathbb{B}^n}}$. Then we must also have
    \begin{equation*}
        t \ge \widetilde{\sup}_{A_\lambda}u,
    \end{equation*}
    and Lemma \ref{lem: estimate of |A_f|} hence yields that
    \begin{equation*}
        |A_f|(A)\lesssim_{n,\lambda} K^{n-1}t^{n-1}+C(n,\Sigma(\R^n\setminus{\overline{\mathbb{B}^n}})).
    \end{equation*}
    Thus, we may set $s=s(n,K,\Sigma,t)$ to be bigger than the right hand side of the above estimate, and the claim follows.
\end{proof}

\section{Proof of Proposition \ref{thm: finite q or finite Dh}}\label{sec: proof of the property}

In this section, we prove Proposition \ref{thm: finite q or finite Dh}. We recall the statement for the convenience of the reader. 

\begin{customprop}{\ref{thm: finite q or finite Dh}}
    Let $K\ge 1$ and $\Sigma\in L^{1+\varepsilon}_{\rm loc}(\R^n\setminus\overline{\B^n})\cap L^1(\R^n\setminus\overline{\B^n})$ for some $\varepsilon>0$. Then there exists a positive integer $q=q(n,K)\in\Z_{>0}$ with the following property: if a continuous mapping $h\in W^{1,n}_{\rm loc}(\R^n\setminus\overline{\B^n},\S^n)$ has a $(K,\Sigma)$-quasiregular value with respect to the spherical metric at $q$ distinct points $\omega_1,...\omega_q\in \S^n\setminus f(\R^n\setminus\overline{\B^n})$, all of which are asymptotic values of $f$ at infinity, then $\abs{Dh}\in{L^n(\R^n\setminus\overline{2\mathbb{B}^n})}$.
\end{customprop}

Our approach again roughly follows the argument in \cite[Section 6]{kangasniemi2024quasiregular}, combining the pseudosupremum estimates obtained in Section \ref{sec: original tools} with conformal capacity estimates and an appropriate hunting lemma.

\subsection{The revised hunting lemma}

As noted in the introduction, our proof requires a revised version of Rickman's hunting lemma. We now state this revised version and provide a proof. The general ideas of the proof remain similar as in \cite[Lemma 2.1]{bonk2019rickman}.

\begin{lem}\label{lem: Rickman hunting lem}
    Let $r_0 > 0$, let $c > 1$, and let $\mu$ be a (non-negative) locally finite Borel measure on $\R^n$ such that $\mu$ has no atoms, $\mu(\R^n) = \infty$, and there exists an $R_0 > 0$ such that
    \[
        \mu(\B^n(0, cr)) \ge 3 \mu(\B^n(0, r))
        \qquad \text{for all } r \ge R_0.
    \]
    Then there exists a constant $D = D(n, c)$ and a sequence of balls $B_j$, $j \in \Z_{>0}$ such that $16B_j \subset \R^n \setminus\overline{\B^n}(0, r_0)$, $\mu(16B_j) \le D \mu(B_j)$, and $\mu(B_j) \to \infty$ as $j \to \infty$.
\end{lem}

\begin{proof}
Let $D=D(n,c)\in\mathbb{N}$ be sufficiently large so that the
following two covering properties hold:

\begin{enumerate}
    \item for every $r>0$, the set
    \(
        \B^n(0,cr)\setminus \B^n(0,r)
    \)
    can be covered by at most $D$ balls of radius $r/64$ whose
    centers belong to $\B^n(0,cr)\setminus \B^n(0,r)$;

    \item every ball of radius $\rho>0$ can be covered by at most
    $D$ balls of radius $\rho/32$ whose centers belong to the
    original ball.
\end{enumerate}
Such a constant clearly exists.

Suppose towards contradiction that there exists $C>0$ such that, for every ball $B$ with $16B\subset\R^n\setminus{\overline{\mathbb{B}^n}}(0,r_0)$ and $\mu(B)>C$,  we have $\mu(16B)>D\mu(B)$. Since $\mu(\mathbb{R}^n)=\infty$, we may select an $r \ge \max \{R_0, 2r_0\}$ such that $\mu(\B^n(0, r)) > CD/2$. 

We cover
$\B^n(0,cr)\setminus \B^n(0,r)$
by balls of radius $\rho_0 := r/64$ as above. Hence, there exists one such
ball $B_0=\B^n(x_0,\rho_0)$
satisfying
\begin{equation*}
    \mu(B_0)\ge
    \frac{1}{D}
    \mu\bigl(\B^n(0,cr)\setminus \B^n(0,r)\bigr) 
    \ge
    \frac{2}{D}\mu\bigl(\B^n(0,r)\bigr)
    > C.
\end{equation*}
We also observe that, by the choice of $x_0$ and $r$, we have $|x_0|\ge r$ and
\[
    |x_0|-32\rho_0
    \ge \frac{r}{2}
    \ge r_0.
\]
Consequently, $16B_0 \subset 32B_0\subset
\mathbb{R}^n\setminus\overline{\B^n}(0,r_0)$, and our hypothesis hence yields that $\mu (16B_{0})>D\mu(B_{0})$.

Thus, $16B_0$ can be covered by $D$ balls of radius $\rho_1 := \rho_0/2$ centered in $16B_0$. Hence, there exists a ball $B_1:=\B^n(x_1,\rho_1)$ satisfying $|x_1-x_0|<16\rho_0$ and
    \begin{equation*}
        \mu(16B_1)\ge \frac{1}{D}\mu(16B_0)>\mu(B_0)>C.
    \end{equation*}
    Moreover, we have
    \[
    |x_0-x_1|+16\rho_1 < 16\rho_0+8\rho_0 < 32\rho_0,
    \] 
    and consequently $16B_1\subset 32 B_0\subset\R^n\setminus{\overline{\mathbb{B}^n}}(0,r_0)$. 
    
    Now we repeat the argument for $B_1$, and so on, decreasing the radii of the balls by the factor 2 in each step. In this way, we obtain a sequence of balls $B_k=\B^n(x_k, \rho_k)$ with $\rho_k := 2^{-k} \rho_0$ satisfying $\mu(B_k)>C$ and $|x_{k+1}-x_k|<2^{4-k}\rho_0$. Moreover, we have 
    \[
        |x_k-x_0|+16\rho_k\le
        \sum_{j=0}^{k-1}|x_{j}-x_{j+1}|+16(2^{-k}\rho_0)
        \le \sum_{j=0}^{k}16(2^{-j}\rho_0)
        <32\rho_0,
    \]
    and consequently $16B_k\subset 32B_0\subset\R^n\setminus{\overline{\mathbb{B}^n}}(0,r_0)$ for all $k$. 
    
    In particular, $x_k$ form a Cauchy sequence, and hence there exists an $x_\infty$ such that $x_k\to x_\infty$ as $k\to\infty$. 
    If $\delta>0$ is arbitrary, then $B_k\subset \B^n(x_\infty,\delta)$ for some $k\in\Z_{>0}$, and hence $\mu( \B^n(x_\infty,\delta))\ge\mu(B_k)>C$. Hence 
    \[
    \mu(\{x_\infty\})=\lim_{\delta\to0}\mu(B(x_\infty,\delta))\ge C>0.
    \]
    Here we used that $\mu(B(x_\infty,\delta))<\infty$ for all $\delta$, as follows from the local finiteness of $\mu$. We obtain a contradiction, because $\mu(\{x_\infty\})>0$ but $\mu$ has no atoms. This completes the proof.
\end{proof}

The following example demonstrates that some sort of extra assumption over the original Rickman's hunting lemma, such as the growth condition $\mu(\B^n(0, cr)) \ge 3 \mu(\B^n(0, r))$ above, is necessary in Lemma \ref{lem: Rickman hunting lem}.

\begin{example}\label{exam: rickman}
Let \(n\ge 2\), and define a Borel measure \(\mu\) on
\(\mathbb R^n\) by
\[
    d\mu(x)=w(x)\,dx,
    \qquad
    w(x)=
    \begin{cases}
        1, & |x|\le 1,\\
        |x|^{-n}, & |x|>1,
    \end{cases}
\]
where $dx$ denotes Lebesgue measure. Note that $\mu$ is absolutely continuous with respect to the Lebesgue measure, \(\mu(E)<\infty\) for every compact set
\(E\subset\mathbb R^n\), and
\[
    \mu(\B^n(0,r))
    =
    \begin{cases}
        \omega_n r^n, & 0\le r\le1,\\
        \omega_n+n\omega_n\log r, & r>1.
    \end{cases}
\]
Consequently, it also follows that \(\mu(\mathbb R^n)=\infty\).

We claim that that for every ball $B = \B^n(x_0, r_0) \subset \R^n$, if $2B \cap \B^n = \emptyset$, then $\mu(B) \le C(n)$. Indeed, we divide into two cases, according to whether $|x_0|\le 2$ or $|x_0|>2$.
If $|x_0|\le 2$, then $r_0 \le 1$ and 
\[
    \mu(B)
    \le \mu(3\B^n)
    = (1 + n\log(3)) \omega_n.
\]
If \(|x_0|>2\), then $2B \cap \B^n = \emptyset$ yields that $r_0 \le \abs{x_0}/2$, and hence every
\(y\in B\) satisfies
\[
    |y|\ge |x_0|-r_0
    \ge\frac{|x_0|}{2}>1.
\]
Thus,
\begin{equation*}
    \mu(B)
    =\int_{B} |y|^{-n}\,dy
    \le\left(\frac{\abs{x_0}}{2}\right)^{-n}\omega_n r_0^n
    \le\omega_n.
\end{equation*}
We have thus shown in each case that $\mu(B) \le C(n)$.
\end{example}

\begin{rem}\label{rem:Rickman's statement}
    In \cite[Lemma 5.1]{rickman1980picard}, Rickman states a hunting lemma as follows: \emph{If $\mu$ is a Borel measure on $\R^n$ such that $\mu(\R^n) = \infty$, $\mu$ is finite on compact sets, and $r \mapsto \mu(r\overline{\B^n})$ is continuous, then for every $\alpha > 1$, there exists a $L = L(n, \alpha) > 0$ and a sequence of balls $B_i = \B^n(x_i, r_i)$ such that $\mu(\alpha B_i) \le L \mu(B_i)$ for all $i$, $\lim_{i \to \infty} \mu(B_i) = \infty$, and moreover
    \[
        \lim_{i \to \infty} x_i = \infty
        \quad \text{and} \quad
        \lim_{i \to \infty} \frac{r_i}{\abs{x_i}} = 0.
    \]}

    This version of the hunting lemma would immediately imply Lemma \ref{lem: Rickman hunting lem}. However, unfortunately the above statement cannot be true due to Example \ref{exam: rickman}. Indeed, the measure $\mu$ in the example satisfies the assumptions in the statement, but for all balls $B = \B^n(x, r) \subset \R^n$ with $\abs{x} > 2$ and $r/\abs{x} < 1/4$, we have $2B \cap \B^n = \emptyset$ and hence $\mu(B) \le C(n)$.

    The error in the proof of \cite[Lemma 5.1]{rickman1980picard} is relatively subtle, and only takes effect at the very end. Namely, the beginning of the proof divides into two cases, where a set of identical estimates is shown in the first case for variables denoted $t, s, s'$, and in the second case for variables denoted $t_{m-1}, t_m, t_{m+1}$. Afterwards, the proof merges both cases under the notation $t, s, s'$ of the first case; however, the notation $s'$ carries extra meaning in the proof that is not present in the second case, and that is necessary to conclude the desired properties for the constructed sequence of balls at the very end of the proof.

    In the proof of the Great Picard theorem presented later by Rickman in his book, see \cite[Theorem IV.2.27]{rickman2012quasiregular}, a hunting lemma as in \cite[Lemma 5.1]{rickman1980picard} is not seemingly used. In fact, the ideas presented in \cite[Proof of Theorem IV.2.27]{rickman2012quasiregular} are closely reminiscent of how Lemma \ref{lem: Rickman hunting lem} is formulated and proven.
\end{rem}

\subsection{Applying the hunting lemma}

Next, we proceed to resolve the case $\|Dh\|_{L^n(\R^n\setminus
\overline{2\mathbb{B}^n})}=\infty$ using Lemma \ref{lem: Rickman hunting lem}. The following technical lemma encapsulates the consequences of applying the proof of \cite[Theorem 1.5]{kangasniemi2024quasiregular}, without the component involving the hunting lemma, on annuli in $\R^n \setminus \overline{\B^n}$.

\begin{lem}\label{lem: sufficiently large volume for annulus}
		Let $K \ge 1$, and let $\Sigma \in L^{1+\eps}_\loc(\R^n \setminus \overline{\B^n}) \cap L^1(\R^n \setminus \overline{\B^n})$ for some $\eps > 0$. Let $h \in W^{1,n}_\loc(\R^n \setminus \overline{\B^n}, \S^n)$ be a continuous map with a $(K, \Sigma)$-quasiregular value at $q$ distinct points $\omega_1, \dots, \omega_q \in \S^n \setminus h(\R^n \setminus \overline{\B^n})$, all of which are asymptotic values of $h$ at infinity. Suppose that $\norm{Dh}_{L^n(\R^n \setminus \overline{2\B^n})} = \infty$, and let $\mu$ denote the signed measure on $\R^n \setminus \overline{\B^n}$ with $d\mu = J_h \vol_n$. 
		
		Let $r_0 \in [0, \infty)$, $\lambda \in (0, 1)$, and let $x_j \in \R^n$ and $R_j \in (r_0, \infty)$ be such that $\mathbb{A}^n(x_j; \lambda r_0, 16R_j) \subset \R^n \setminus \overline{2\B^n}$ for all $j \in \N$ and $\smallabs{\mu}(\mathbb{A}^n(x_j; r_0, R_j)) \to \infty$ as $j \to \infty$. Then there exists a $j_0 \ge 0$ and a sequence of $L_j \in [1, \infty)$ tending to infinity such that, for all $j \ge j_0$, we have the estimates
		\begin{align*}
			\smallabs{\mu}(\mathbb{A}^n(x_j; r_0, R_j)) 
			&\lesssim_{n,\lambda} K^{n-1} L_j^{n-1} + C(n, \Sigma), \qquad \text{and}\\
			L_j^n q^{\frac{1}{n-1}} 
			&\lesssim_{n,\lambda} L_j^n K^n + L_j K \smallabs{\mu}(\mathbb{A}^n(x_j; \lambda r_0, 16 R_j))
			+ C(n, \Sigma).
		\end{align*}
\end{lem}

\begin{proof}
    For every point $\omega_k$, we select a rotation $R_k:\S^n\to\S^n$ that takes $\omega_k$ to $s_n(\infty)$, and denote $h_k=R_k\circ h$. Since $R_k$ is an orientation-preserving isometry of $\S^n$, it follows that $h_k$ has a $(K,\Sigma)$-quasiregular value with respect to the spherical metric at $s_n(\infty)$. Moreover, $s_n(\infty)\notin h_k(\R^n\setminus
    \overline{\B^n})$, and hence $f_k:=s_n^{-1}\circ h_k$ define continuous maps $f_k\in W^{1,n}_{\rm loc}(\R^n\setminus\overline{\B^n},\R^n)$. Notably, for every $k\in\{1,\cdots,q\}$ and every measurable $E\subset\R^n\setminus\overline{\B^n}$, we have 
    \begin{equation*}
        A_{f_k}(E) =\int_E h_k^* \vol_{\S^n}=\int_E h^*R_k^*\vol_{\S^n}=\int_E h^*\vol_{\S^n} = \mu(E).
    \end{equation*}
    Thus, every $A_{f_k}$ is in fact the measure $\mu$ fixed in the lemma statement. 
    We also note that since every $h_k$ has a $(K,\Sigma)$-quasiregular value with respect to the spherical metric at $s_n(\infty)$, we obtain that every $f_k$ satisfies \eqref{dis: Eucl interp of sphe qr values} by Lemma \ref{lem: item comparison 3}, allowing us to use the results of Section \ref{sec: original tools} on $f_k$.

    We then let $u_k=v_{\R^n}\circ f_k$ for every $k\in\{1,...,q\}$. We note that the sets $s_n(\{\infty\}\cup \{v_{\R^n}>t\})$ form a neighborhood basis of $s_n(\infty)$, with the neighborhoods becoming smaller as $t$ increases. Hence, there exists $C_0=C_0(n,\omega_1,...,\omega_q)>0$ such that for every $t\ge C_0$, the sets $\{u_k>t\}$ are pairwise disjoint. Since each $u_k$ is continuous on $\partial2\mathbb{B}^n$, we have 
    $\max_{\partial 2\B^n}u_k<\infty$ for every $k\in\{1,\cdots,q\}$. We may set 
    \[
        t_0:=3\max\left(C_0, \max_{1\le k\le q}\max_{\partial2\mathbb B^n}u_k\right)+1.
    \]
   For each $j\in\N$, we apply Lemma~\ref{lem: unbounded components} to to $f_k$ with $x_0=x_j$, inner radius $r_0$, outer radius $R_j$, $t=t_0$, and  any fixed $\lambda$. Since $t_0>\max_{\partial2\B^n}u_k$ and $A_{f_k}=\mu$ for every $k\in\{1,\cdots,q\}$, Lemma~\ref{lem: unbounded components} yields that there exists an $s_0=s_0(n,K,\Sigma,t_0)>0$ such that
    $\mathbb A^n(x_j; \lambda r_0,2R_j)$ meets an unbounded component of
    each set $\{u_k>t_0\}$ contained in
    $\R^n\setminus\overline{2\B^n}$ whenever
    $\smallabs{\mu}(\mathbb A^n(x_j;r_0,R_j))>s_0$. Since $\smallabs{\mu}(\mathbb A^n(x_j;r_0,R_j))\to\infty$, we may choose $j_0\in\N$ such that this conclusion holds for every $j\ge j_0$.
    
    For such $j$, we set
    \begin{equation*}
        L_{j,k}:=\widetilde{\sup}_{\A^n(x_0;  \lambda r_0,2R_j)}u_k.
    \end{equation*}
    In particular, the preceding conclusion implies $L_{j,k}\ge t_0$ for every $k\in\{1,\cdots,q\}$. By applying Lemma \ref{lem: estimate of |A_f|} to $f_k$ with the same center $x_0=x_j$, the same radii $r_0$ and $R_j$, and our given $\lambda$, we obtain that
    \begin{equation}\label{ineq: upper bound of |mu|(A(3,j))}
        |\mu|(\A^n(x_0;r_0,R_j))\lesssim_{n,\lambda} K^{n-1}L^{n-1}_{j,k}+C(n,\Sigma).
    \end{equation}
    In fact, we now have the first of our desired estimates for any selection of $L_j = L_{j, k(j)}$.
    
    It follows from \eqref{ineq: upper bound of |mu|(A(3,j))} that
    \begin{equation}\label{eq: L_j,k to infty}
        \lim_{r\to\infty}
        \min_{k}L_{j,k}=\infty.
    \end{equation}
    Set $\Omega:=\R^n\setminus\overline{2\B^n}$. We define
    \begin{equation*}
        \begin{aligned}
        E^r_{k}=\{x\in \Omega:u_k\ge 2L_{j,k}/3\},\quad
        F^r_{k}=\{x\in \Omega: u_k\le L_{j,k}/3\},
        \end{aligned}
    \end{equation*}
    and
    \begin{equation*}
        E^r_{k,2}=E^r_{k}\cap\overline{\A^n}(x_0;2R_j,4R_j),\qquad F^r_{k,2}=F^r_{k}\cap\overline{\A^n}(x_0;2R_j,4R_j).
    \end{equation*}
    The sets $E_k^r,F_k^r$ are relatively closed in $\Omega$, and $E_k^r\cap F_k^r=\emptyset$. Moreover, since $L_{j,k}/3\ge t_0/3>C_0$, the sets
    \[
        \Omega\setminus F_k^j
        =\{x\in\Omega: u>L_{j,k}/3\},\quad k\in\{1,\cdots,q\},
    \]
    are pairwise disjoint. Thus $F_k^j\cup F_l^j=\Omega$ whenever $k\ne l$. By the definition of $L_{j,k}$, the annulus $\A^n(x_0;\lambda r_0,2R_j)$ meets an unbounded component of the set
    $\{u_k>2L_{j,k}/3\}$. Since 
    \[
        2L_{j,k}/3\ge 2t_0/3>\max_{1\le k\le q}\max_{\partial2\mathbb B^n}u_k,
    \]
    this component is contained in $\Omega$. Consequently, $\B^n(x_0,2R_j)$ meets an unbounded component of $E_k^j$ for every $k\in\{1,\cdots,q\}$. We may therefore apply Lemma \ref{lem: condenser capacity estimate lem} with $B=\B^n(x_0,2R_j)$ and $t=2$ to obtain
     \begin{equation}\label{ineq: capacity and omit q}
        \sum_{k=1}^q\capac(E^j_{k,2},F^j_{k,2})\gtrsim_n q^{\frac{n}{n-1}}.
    \end{equation}

    We may fix a function $\psi_j\in C^\infty_0(\A^n(x_0; r_0,8R_j),[0,1])$ for which $\psi_j\equiv1$ on a neighborhood of $\overline{\A^n}(x_0;2R_j,4R_j)$ and $\|\nabla\psi_j\|_{L^n}\lesssim_{n,} 1$. Now, the function
    \begin{equation*}
        \eta_j=\left(\frac
        {3\min(u_k,L_{j,k})}{L_{j,k}}-1\right)\psi_j
    \end{equation*}
    is admissible for the condenser $(E^j_{k,2},F^j_{k,2})$ in $\Omega$. It follows that 
    \begin{equation*}
        \capac(E^j_{k,2},F^j_{k,2})
        \le\int_{\Omega}
        |\nabla\eta_j|^n
        \lesssim_{n}\|\nabla\psi_j\|^n_{L^n}+\int_{\A^n(x_0; r_0,8R_j)\cap\{u_k<L_{j,k}\}}\frac{|\nabla u_k|^n}{L^n_{j,k}}.
    \end{equation*}
    We apply Corollary \ref{cor: truncate estimate of |du|^n on u<L} to the last integral with $x_0=x_j$, inner radius $r_0$, outer radius $8R_j$, and our given $\lambda$. By also using $\|\nabla\psi_j\|^n_{L^n}\lesssim_n 1 \le K^n$, we obtain that
    \begin{equation*}
        \capac(E^j_{k,2},F^j_{k,2})\lesssim_{n,\lambda} K^n+\frac{K|\mu|(\A(\lambda r_0, 16R_j))}{L^{n-1}_{j,k}}+\frac{C(n,\Sigma)}{L^{n}_{j,k}}.
    \end{equation*}
    By using \eqref{ineq: capacity and omit q}, there always exists an index $k_0 = k_0(j) \in\{1,\cdots,q\}$ such that 
    \begin{equation*}
        \capac(E^j_{k_0,2},F^j_{k_0,2})\gtrsim_n q^{\frac{1}{n-1}}.
    \end{equation*}
    Hence, for this specific choice of $k_0$, we have 
    \begin{equation}\label{eq:q_upper_bound}
        q^{\frac{1}{n-1}}\lesssim_n K^n+\frac{K|\mu|(\A^n(\lambda r_0,16R_j))}{L^{n-1}_{j,k_0}}+\frac{C(n,\Sigma)}{L^{n}_{j,k_0}}.
    \end{equation}
    Finally, set $L_j:=L_{j,k_0}$. Then $L_j\to\infty$ by \eqref{eq: L_j,k to infty}, the first desired estimate holds as observed in \eqref{ineq: upper bound of |mu|(A(3,j))}, and the second desired estimate follows by multiplying both sides of \eqref{eq:q_upper_bound} by $L^n_{j,k_0}$.
\end{proof}

We are now ready to prove Proposition \ref{thm: finite q or finite Dh} by combining Lemmas \ref{lem: Rickman hunting lem} and \ref{lem: sufficiently large volume for annulus}. The idea is to divide into two cases. In the case where the Hunting Lemma \ref{lem: Rickman hunting lem} applies, we apply Lemma \ref{lem: sufficiently large volume for annulus} on the balls it yields to obtain the claim. In the case where Lemma \ref{lem: Rickman hunting lem} does not apply, the claim instead follows by applying Lemma \ref{lem: sufficiently large volume for annulus} on annuli centered at the origin.

\begin{proof}[{Proof of Proposition \ref{thm: finite q or finite Dh}}]
    Let $h\in W^{1,n}_{\text{loc}}(\R^n\setminus\overline{\B^n},\S^n)$ be a continuous map with a $(K,\Sigma)$-quasiregular value with respect to the spherical metric at $q$ distinct points $\omega_1,\cdots,\omega_q\in \S^n\setminus h(\R^n\setminus\overline{\B^n})$, all of which are asymptotic values of $h$ at infinity. Suppose that $\|Dh\|_{L^n(\R^n\setminus\overline{2\mathbb{B}^n})}=\infty$. Our objective is hence to find an upper bound on $q$ that only depends on $n$ and $K$.

   Since $h\in W^{1,n}_{\loc}(\R^n\setminus\overline{\B^n},\S^n)$, we have $|Dh|\in L^n(\A^n(2,3))$. Hence, the assumption implies that $\|Dh\|_{L^n(\R^n\setminus\overline{3\mathbb{B}^n})}=\infty$. Since moreover $\Sigma(\R^n\setminus\overline{3\mathbb{B}^n})<\infty$, \eqref{defn: qr values into S^n} yields that
    \begin{multline*}
        |\mu|(\R^n\setminus\overline{3\mathbb{B}^n})\ge\mu(\R^n\setminus\overline{3\mathbb{B}^n})=\int_{\R^n\setminus\overline{3\mathbb{B}^n}}J_h\\
        \ge\frac{1}{K}\left(\int_{\R^n\setminus\overline{3\mathbb{B}^n}}|Dh|^n-\pi^n\int_{\R^n\setminus\overline{3\mathbb{B}^n}}\Sigma\right)=\infty.
    \end{multline*}

    Our objective is to find an $r_0 \ge 0$, a $\lambda = \lambda(n) \in (0, 1)$ and sequences of $x_j \in \R^n$ and $R_j \ge r_0$ such that $\mathbb{A}^n(x_j; \lambda r_0, 16R_j) \subset \R^n \setminus \overline{2\B^n}$ for all $j \in \N$, $\smallabs{\mu}(\A^n(x_j;r_0,R_j))\to\infty$, and 
    \[ 
        |\mu|(\A^n(x_j;\lambda r_0,16R_j))\lesssim_n|\mu|(\A^n(x_j;r_0,R_j)) 
    \]
    We first point out how the existence of these implies the claim. Indeed, by applying Lemma \ref{lem: sufficiently large volume for annulus}, there exists a $j_0\ge0$ and a sequence of $L_j\in[1,\infty)$ tending to infinity such that 
    \begin{multline*}
        L_j^n q^{\frac{1}{n-1}}\lesssim_nL_j^nK^n+L_jK|\mu|(\A^n(x_0;\lambda r_0,16R_j))+C(n,\Sigma)\\
        \lesssim_n L_j^nK^n+L_j K|\mu|(\A^n(x_0;r_0,R_j))+C(n,\Sigma)\\
        \lesssim_n K^nL_j^n+L_jK C(n,\Sigma) + C(n, \Sigma)
    \end{multline*}
    for all $j\ge j_0$. Since $L_j\to\infty$ and $C(n,\Sigma)$ is finite, dividing both sides by $L_j^n$ and letting $j\to\infty$, we obtain 
    \[
        q\lesssim_n K^{n(n-1)}.
    \]

    It remains to find $r_0$, $\lambda$, $x_j$, and $R_j$. We divide the proof into two cases. First, we suppose that 
    \begin{equation}\label{ineq: A(r,16r) ge 2A(3,r)}
        |\mu|(\A^n(r,16r))\ge
        2|\mu|(\A^n(3,r))
    \end{equation}
    for all sufficiently large $r$. We define a Borel measure $\nu$ on $\R^n$ by $\nu:=|\mu|\llcorner \bigl(\mathbb{R}^n\setminus \overline{3\mathbb{B}^n}\bigr)$. Then $\nu$ is a nonnegative, locally finite Borel measure with no atoms, and
    \[
        \nu(\R^n)=|\mu|(\R^n\setminus\overline{3\B^n})=\infty.
    \]
    Moreover, for all sufficiently large $r$, \eqref{ineq: A(r,16r) ge 2A(3,r)} yields that
    \[
        \nu\bigl(\B^n(0,16r)\bigr)=
        |\mu|\bigl(\A^n(3,16r)\bigr)
        \ge3|\mu|\bigl(\A^n(3,r)\bigr)=
        3\nu\bigl(\B^n(0,r)\bigr).
    \]
    Hence, Lemma \ref{lem: Rickman hunting lem}, applied to $\nu$ with $c=16$ and $r_0=3$, yields a constant $D=D(n)>1$ and a sequence of balls $B_j$, $j\in\mathbb{Z}_{>0}$, such that
    \begin{equation}\label{eq: hunting balls theorem}
        16B_j\subset\R^n
        \setminus\overline{3\B^n},
        \quad
        |\mu|(16B_j)\le D|\mu|(B_j),
        \quad\text{and}\
        |\mu|(B_j)\to\infty.
    \end{equation}
    In this case, we set $r_0 = 0$, $\lambda = 2/3$, and select $x_j$, $R_j$ by $B_j=\B^n(x_j,R_j)$.

    We then consider the other case, in which there exists a sequence $r_j\to\infty$ such that 
    \[
        \abs{\mu}(\A^n(r_j,16r_j))<
        2\abs{\mu}(\A^n(3,r_j)).
    \]
    Since $|\mu|$ is absolutely continuous with respect to Lebesgue measure, all annulus boundaries have $\mu$-measure zero. Since also $\abs{\mu}(\A^n(3,r_j)) \to \infty$ as $j \to \infty$, it follows by the additivity of $\mu$ that
    \[
       \abs{\mu}(\A^n(2,16r_j))<
        3\abs{\mu}(\A^n(3,r_j))+\abs{\mu}(\A^n(2,3))<4\abs{\mu}(\A^n(3,r_j))
    \]
    for all sufficiently large $j$. Thus, after changing the starting index to be sufficiently large, we may set $x_j=0$, $R_j=r_j$, $r_0=3$, and $\lambda=2/3$, completing the proof.
\end{proof}

\section{The asymptotic value version of the Great Picard theorem}\label{sec: picard type thms sec}

By Proposition \ref{thm: finite q or finite Dh}, the desired upper bound for $q$ has already been obtained in the case $\|Dh\|_{L^n(\R^n\setminus
\overline{2\B^n})}=\infty$. In this section, we treat the remaining case $\|Dh\|_{L^n(\R^n\setminus
\overline{2\B^n})}<\infty$ and complete the proof of the asymptotic value version of our main theorem.

We begin by pointing out that for continuous $h\in W^{1,n}_{\rm loc} (\R^n\setminus\overline{\B^n},\S^n)$ that omit at least one quasiregular value, we can still use the standard higher integrability techniques of the theory of quasiregular values. The a-priori continuity and presence of an omitted point are critical since the Caccioppoli inequality has no analogue for maps into $\S^n$, unlike the case of maps into a punctured $\S^n$.

\begin{lem}\label{lem: higher integrabilitiy of G}
    Let $K\ge 1$, let $\omega_0 \in \S^n$, and let $\Sigma\in L^{1+\varepsilon}_{\rm loc}(\R^n\setminus\overline{\B^n})$ for some $\varepsilon>0$. Suppose that $h\in W^{1,n}_{\rm loc} (\R^n\setminus\overline{\B^n},\S^n\setminus\{\omega_0\})$ is continuous and has a $(K,\Sigma)$-quasiregular value with respect to the spherical metric at $\omega_0$. Then there exists $\varepsilon'\in(0,\varepsilon)$ such that $|Dh|\in L^{n(1+\varepsilon')}_\loc(\R^n\setminus\overline{\B^n})$.
\end{lem}

\begin{proof}
    By post-composing $h$ with an isometric rotation, we may assume that $\omega_0 = s_n(\infty)$. Because $h$ is continuous and $s_n^{-1}$ is locally Lipschitz in $\S^n \setminus s_n(\infty)$, it follows that $f := s_n^{-1} \circ h \in W^{1,n}_\loc(\R^n\setminus\overline{\B^n}, \R^n)$. 

    By Lemma \ref{lem: item comparison 3}, $f$ satisfies
    \[
        |Df(x)|^n\le KJ_f(x)+\widetilde\Sigma(x)
    \]
    for a.e. $x\in\R^n\setminus{\overline{\B^n}}$, where $\widetilde\Sigma:=C(n)(1+|f|^2)^{\frac{n}{2}}\Sigma$ and $C(n)$ is the constant  appearing in Lemma \ref{lem: item comparison 3}. Since $f$ is locally bounded, it follows that $\widetilde{\Sigma}\in L^{1+\varepsilon}_{\loc}(\R^n\setminus{\overline{\B^n}})$. Thus, by \cite[Theorem 1.4]{dolevzalova2024mappings}, it follows that $|Df|\in L^{n(1+\varepsilon')}_\loc(\R^n\setminus\overline{\B^n})$ for some $\eps' \in (0, \eps)$. 
    Finally, by \eqref{eq: Dh and J_h}, 
    \[
        |Dh|=\frac{2|Df|}{1+|f|^2}\le 2|Df|,
    \]
    which proves that $|Dh|^n\in L^{1+\varepsilon'}_{\loc}(\R^n\setminus{\overline{\B^n}})$.
\end{proof}

Our objective is now to prove Proposition \ref{prop: extend to R^n}. We first recall the statement. 

\begin{customprop}{\ref{prop: extend to R^n}}
    Let $K\ge 1$ and $\Sigma\in L^{1\pm\varepsilon}(\R^n\setminus\overline{\B^n})$ for some $\varepsilon>0$. Let $h\in W^{1,n}_{\rm loc} (\R^n\setminus\overline{\B^n},\S^n)$ be continuous with $\abs{Dh} \in L^n(\R^n\setminus\overline{2\B^n})$. Suppose that $h$ has a $(K,\Sigma)$-quasiregular value with respect to the spherical metric at $q$ distinct points $\omega_1,...\omega_q\in \S^n\setminus h(\R^n\setminus\overline{\B^n})$, all of which are asymptotic values of $h$ at infinity. 
    
    Then there exist a $\widetilde{\Sigma} \in L^{1\pm\varepsilon'}(\R^n)$ for some $\varepsilon'\in(0,\varepsilon)$, and a continuous $\tilde{h}\in W^{1,n}_{\rm loc} (\R^n,\S^n)$ such that $\smallabs{D\tilde{h}} \in L^n(\R^n)$ and $\tilde{h}$ has a $(K,\widetilde{\Sigma})$-quasiregular value with respect to the spherical metric at $q-1$ distinct points $\tilde{\omega}_1,...\tilde{\omega}_{q-1}\in \partial \tilde{h}(\R^n)$.
\end{customprop}

We start with the following lemma which states that, given a continuous map $h \in W^{1,n}_\loc(\R^n \setminus \overline{\B^n}, \S^n)$ with $\norm{Dh}_{L^n(\R^n \setminus \overline{\B^n})} < \infty$, there exists a ``good'' radius $r_0 > 1$ such that $h(\partial\B^n(0, r_0))$ is small.

\begin{lem}\label{lem: good radius}
    Let $d > 0$, let $h \in W^{1,n(1+\eps)}_\loc(\R^n \setminus \overline{\B^n}, \S^n)$ be continuous with $\eps > 0$, and suppose that $\norm{Dh}_{L^n(\R^n \setminus \overline{2\B^n})} < \infty$. Then there exists an $r_0>1$ satisfying the following properties:
    \begin{enumerate}[label=\textup{(\arabic*)}]
    \item $h\vert_{\partial \B^n(0, r_0)}\in W^{1,n(1+\eps)}(\partial \B^n(0, r_0),\S^n)$;
    
    \item $h(\partial\B^n(0,r_0))\subset B_{\S^n}(a,d)$ for some $a\in\S^n$.
    \end{enumerate}
\end{lem}
\begin{proof}
    We embed $\S^n$ into $\R^{n+1}$ in the standard manner. 
    We first show that $r_0$ can be chosen so that
    $h(\partial\mathbb{B}(0,r_0))\subset B_{\S^n}(a,d)$ for some $a\in\S^n$. Since $|Dh|\in{L^n(\R^n\setminus
    \overline{2\mathbb{B}^n})}$, for any constant $C_0$, to be fixed below, we may choose $\rho>2$ sufficiently large so that
    \begin{equation*}
        \int_{\rho}^{2\rho}\int_{\partial\B^n(0,r)}|Dh|^n\,{\rm d}\mathcal{H}^{n-1}\,{\rm d}r=\int_{\A^n(\rho,2\rho)}|Dh|^n\le C_0\,d^n.
    \end{equation*}
    It follows that there exists a set $\mathcal{R}\subset[\rho,2\rho]$ of 
    positive measure such that 
    \begin{equation}\label{ineq: Dh less than d^n}
        r\int_{\partial\B^n(0,r)}|Dh|^n\,{\rm d}\mathcal{H}^{n-1}\le \frac{C_0\, d^n}{\ln 2}
    \end{equation}
    for every $r\in\mathcal{R}$. Since $h\in W^{1,n}(\A^n(\rho,2\rho),\S^n)$ and $h$ is continuous, the Sobolev slicing theorem gives
    \[
        h\vert_{\partial \B^n(0, r)}\in W^{1,n(1+\eps)}(\partial \B^n(0, r),\S^n)\cap C(\partial \B^n(0, r),\S^n)
    \]
    for a.e.\ $r\in[\rho,2\rho]$. We may therefore choose an $r_0\in\mathcal{R}$ such that 
    \[
        h\vert_{\partial \B^n(0, r_0)}\in W^{1,n(1+\eps)}(\partial \B^n(0, r_0),\S^n)\cap C(\partial \B^n(0, r_0),\S^n), 
    \]
    and \eqref{ineq: Dh less than d^n} holds with $r=r_0$. 

    We then apply the spherical version of Morrey's inequality to $h\vert_{\partial \B^n(0, r_0)}$, see e.g.\ \cite[Lemma 2.19]{HenclKoskela_book}. We obtain that   
    \[
        \dist_{\S^n}(h(x),h(y))\lesssim_n
        \abs{h(x) -h(y)} \lesssim_n 
        r_0^\frac{1}{n}\left(\int_{\partial\B^n(0,r_0)}|Dh|^n\,{\rm d}\mathcal{H}^{n-1}\right)^{\frac{1}{n}}
    \]
    for all $x,y\in \partial \B^n(0, r_0)$. Hence, for all such $x, y$, it follows that 
    \begin{equation*}
    \dist_{\S^n}(h(x),h(y)) \le C_1(n)\left(\frac{C_0}{\ln 2}\right)^{\frac{1}{n}}d.
    \end{equation*}
    Thus, the claim holds if we choose $C_0=C_0(n)>0$ to be small enough that
    \[
        C_1(n)\left(\frac{C_0}{\log 2}\right)^{1/n}<1.
    \]
\end{proof}

We also need the following classical trace extension result. 

\begin{lem}\label{lem:trace extension}
    Let $n,m\ge 2$ and $1\le p<\infty$. If $B\subset\R^n$ is a ball and
    $f\in (C\cap W^{1,p})(\partial B,\R^m)$, then there exists
    $F\in C(\overline{B},\R^m)\cap W^{1,p}(B,\R^m)$ such that
    \[
        F|_{\partial B}=f
        \quad\text{and}\quad
        \norm{F}_{L^\infty(B)}
        \le\norm{f}_{L^\infty(\partial B)}.
    \]
\end{lem}

\begin{proof}
    By translation and scaling, we may assume that $B=\B^n(0,1)$.
    Define
    \begin{equation*}
        F(x):=
        \begin{cases}
            |x|f\left(\dfrac{x}{|x|}\right),
            & x\in\overline{B}\setminus\{0\},\\[1mm]
            0, & x=0.
        \end{cases}
    \end{equation*}
    By applying for example \cite[Lemma 3]{brezis2019radial} componentwise to $F$ with $a=s=1$, we have $F\in W^{1,p}(B,\R^m)$.
    Since $f$ is continuous and
    \[
        |F(x)|\le |x|\norm{f}_{L^\infty(\partial B)}
        \qquad\text{for all }x\in\overline{B},
    \]
    we have $F\in C(\overline{B},\R^m)\cap W^{1,p}(B,\R^m)$,
    \[
        \norm{F}_{L^\infty(\overline{B})}
        \le \norm{f}_{L^\infty(\partial B)},
        \quad\text{and}\quad
        F|_{\partial B}=f.
    \]
\end{proof}

What remains is the core, technical part of the proof. The idea is that, given a continuous $h\in W^{1,n}_{\rm loc} (\R^n\setminus\overline{\B^n},\S^n)$ such that $h(\partial\B^n(0, r_0))$ is small, we can replace $h$ inside $\B^n(0, r_0)$ with the aforementioned trace extension. We then prove that for small enough $h(\partial\B^n(0, r_0))$, this process preserves all but one of the asymptotic quasiregular values of $h$ at infinity.

\begin{proof}[Proof of Proposition \ref{prop: extend to R^n}]
    We begin the proof by setting
    \begin{equation*}
        d:=\frac{1}{10}\min_{q\neq l}\dist_{\S^n}(\omega_q,\omega_l)>0.
    \end{equation*}
    By Lemma \ref{lem: higher integrabilitiy of G}, we have $h\in W^{1,p}_{\rm loc} (\R^n\setminus\overline{\B^n},\S^n)$ for some $p > n$. Thus, by Lemma \ref{lem: good radius}, we find an $r_0 > 1$ with $h\vert_{\partial \B^n(0, r_0)}\in W^{1,p}(\partial \B^n(0, r_0),\S^n)$ and $h(\partial\B^n(0,r_0))\subset B_{\S^n}(a,d)$ for some $a\in\S^n$.

    Choose $\omega_{k_0}$ among $\{\omega_1,\cdots,\omega_q\}$ such that
    \[
        \dist_{\S^n}
        (a,\omega_{k_0})=\min_{1\le k\le q}\dist_{\S^n}(a,\omega_k),
    \]
    and let $\mathcal{W}$ be the set of the remaining omitted points. Then $\#\mathcal{W}=q-1$. Moreover, for every $\omega_i\in\mathcal{W}$, the definition of $d$ gives
    \begin{equation*}
        \begin{aligned}
            10d\le \dist_{\S^n}(\omega_i,\omega_{k_0})\le\dist_{\S^n}(\omega_i,a)+\dist_{\S^n}(a,\omega_{k_0})\le 2\dist_{\S^n}(\omega_i,a).
        \end{aligned}
    \end{equation*}
    Consequently,
    \begin{equation}\label{ineq: dist W and B(a,d)}
        \dist_{\S^n}\bigl(B_{\S^n}(a,d),\mathcal W\bigr)
        \ge 5d.
    \end{equation}
    Since $d<\pi$ which is the injectivity radius of $\S^n$ at $a$, we may fix a well defined exponential map
    \[
        \exp_a^{-1}:B_{\S^n}(a,d)\to \B^n(0, d) \subset \R^n.
    \]
    
    Since $h(\partial\B^n(0,r_0))\subset B_{\S^n}(a,d)$, the composition $f = \exp_a^{-1} \circ h\vert_{\partial \B^n(0, r_0)}$ is in $(W^{1,p} \cap C)(\partial \B^n(0, r_0), \B^n(0, d))$ by a chain rule of Sobolev and Lipschitz maps. Thus, we may apply the trace extension result from Lemma \ref{lem:trace extension} on it, which provides us a map $F \in (C \cap W^{1,p})(\overline{\B^n}(0, r_0), \B^n(0, d))$ with $F \vert_{\partial \B^n(0, r_0)} = f$. Now, we set
    \[
        H \colon \overline{\B^n}(0, r_0) \to B_{\S^{n}}(a, d), \quad H = \exp_a \circ F,
    \]
    and define
    \begin{equation*}
        \widetilde{h}(x):=\left\{
        \begin{array}{ll}
             H(x),  \quad&{\rm if}\,|x|\le r_0,\\
             h(x),\quad&{\rm if}\,|x|> r_0.
        \end{array}
        \right.
    \end{equation*}

    Since $H\vert_{\partial \B^n(0, r_0)} = h\vert_{\partial \B^n(0, r_0)}$, the map $\tilde{h}$ is continuous. Moreover, since $H \in W^{1,p}(\B^n(0, r_0), \S^n)$ and $h \in W^{1,n}_\loc(\R^n, \S^n)$ have the same trace on $\partial \B^n(0, r_0)$, the Sobolev gluing lemma yields that $\tilde{h} \in W^{1,n}_\loc(\R^n, \B^n)$.

    We now construct a function $\widetilde{\Sigma}\in L^{1\pm\varepsilon'}(\R^n)$  such that 
    $\widetilde{h}$ has a $(K,\widetilde{\Sigma})$-quasiregular value at every $\omega_i\in\mathcal{W}$ with respect to the spherical metric. Define $\widetilde{\Sigma}:\R^n\to\R$ by 
    \begin{equation*}
        \widetilde{\Sigma}(x):=\left\{
        \begin{array}{ll}
             \dfrac{1}{5^nd^n}\left(|DH(x)|^n+K|J_H(x)|\right),\quad&{\rm if}\, x\in\overline{\B^n}(0,r_0),\\
             \Sigma(x),\quad&{\rm if}\, x\in\R^n\setminus \overline{\B^n}(0,r_0).
        \end{array}
        \right.
    \end{equation*}
    Note that $\tilde{\Sigma} \in L^{1+\eps'}(\R^n)$ for some $\eps' > 0$, since $\Sigma \in L^{1+\eps}(\R^n \setminus \B^n)$ with $\eps > 0$, and $DH \in L^p(\B^n(0, r_0))$ with $p > n$.
    
    Since $\widetilde{h}=h$ and $\widetilde{\Sigma}=\Sigma$ outside $\B^n(0,r_0)$, the required inequality already holds on $\R^n\setminus{\overline{\B^n}}(0,r_0)$ for every $\omega_i\in\mathcal W$. It therefore suffices to prove that
    \begin{equation}\label{ineq: qr values for H}
        |DH(x)|^n\le KJ_H(x)+\dist^n_{\S^n}(H(x),\omega_i)\widetilde{\Sigma}(x)
    \end{equation}
    for a.e. $x\in \B^n(0,r_0)$ and for every $\omega_i\in\mathcal{W}$. By \eqref{ineq: dist W and B(a,d)} and the inclusion $H\bigl(\overline{\B^n}(0,r_0)\bigr)\subset B_{\S^n}(a,d)$, we obtain
     \begin{equation*}
        5d\le\dist_{\S^n}(B_{\S^n}(a,d),\mathcal{W})\le\dist_{\S^n}(H(x),\omega_i).
    \end{equation*}
    Consequently,
    \begin{equation*}
        \dist^n_{\S^n}(H(x),\omega_i)\widetilde{\Sigma}(x)\ge |DH(x)|^n+K|J_H(x)|
    \end{equation*}
    for $x\in \B^n(0,r_0)$ and for every $\omega_i\in\mathcal{W}$, which completes the proof of \eqref{ineq: qr values for H}. 

    Finally, because $\widetilde h=h$ in a neighborhood of infinity, every $\omega_i\in\mathcal W$ remains an asymptotic value of $\widetilde h$ at infinity. Moreover, none of the points $\omega_i \in \mathcal{W}$ is in the image of $H$, since none of them are in $B_{\S^n}(a, d)$ by \ref{ineq: dist W and B(a,d)}. Thus, every $\omega_i \in \mathcal{W}$ is also omitted by $\tilde{h}$, showing that $\mathcal{W} \subset \partial \tilde{h}(\R^n)$ and completing the proof.
\end{proof}

With Proposition \ref{prop: extend to R^n} established, we may apply \cite[Theorem 1.4]{kangasniemi2024quasiregular} to get a step closer to our main theorem.

\begin{prop}\label{prop:great_picard_for_asymptotic_values}
    Let $K\ge 1$ and $\Sigma\in L^{1\pm\varepsilon}(\R^n\setminus\overline{\B^n})$ for some $\varepsilon>0$. Then
    there exists a positive integer $q=q(n,K)\in\Z_{>0}$ such that no continuous map $h\in W^{1,n}_{\rm loc}(\R^n\setminus\overline{\B^n},\S^n)$ has a $(K,\Sigma)$-quasiregular value with respect to the spherical metric at $q$ distinct points $\omega_1,\ldots,\omega_q\in \S^n\setminus h(\R^n\setminus\overline{\B^n})$, all of which are asymptotic values of $h$ at infinity.
\end{prop}

\begin{proof}
    By Proposition \ref{thm: finite q or finite Dh}, we must have $q\le q(n,K)$ for some constant $q(n,K)$, or $|Dh|\in{L^n(\R^n\setminus\overline{2\mathbb{B}^n})}$. Suppose then that we are in the latter case. Then Proposition \ref{prop: extend to R^n} yields that there exist a function $\widetilde{\Sigma}\in L^{1\pm\varepsilon'}(\R^n)$ for some $\varepsilon'\in(0,\varepsilon)$, 
    and a continuous map $\widetilde{h}\in W^{1,n}_{\rm loc}(\R^n,\S^n)$ such that $\widetilde{h}$ has a $(K,\widetilde{\Sigma})$-quasiregular value at $q-1$ distinct points $\widetilde{\omega}_{1}\cdots\widetilde{\omega}_{q-1}\in\partial \widetilde{h}(\R^n)$. By \cite[Theorem 1.4]{kangasniemi2024quasiregular}, we obtain that $q-1\le q'(n,K)$ for some constant depends on $n$ and $K$. Hence, by setting $q(n,K) := q'(n, K) + 1$, the proof is complete.
\end{proof}

We also point out that a version of Proposition \ref{prop:great_picard_for_asymptotic_values} with a target of $\R^n$ instead of $\S^n$ follows immediately.

\begin{cor}\label{cor: great picard type thm into R^n}
    Let $K\ge 1$ and $\Sigma\in L^{1\pm\varepsilon}(\R^n\setminus\overline{\B^n})$ for some $\varepsilon>0$. There exists a positive integer $q=q(n,K)\in\Z_{>0}$ such that no continuous map $f\in W^{1,n}_{\loc}(\R^n\setminus\overline{\B^n},\R^n)$ has a $(K,\Sigma)$-quasiregular value at $q$ distinct points $y_1,...,y_q\in \R^n\setminus f(\R^n\setminus\overline{\B^n})$, all of which are asymptotic values of $f$ at infinity.    
\end{cor}

\begin{proof}
     Suppose that $f\in W^{1,n}_{\text{loc}}(\R^n\setminus{\overline{\B^n}},\R^n)$ is continuous and has a $(K,\Sigma)$-quasiregular value at $q$ distinct points $y_1,...,y_q\in \R^n\setminus f(\R^n\setminus\overline{\B^n})$, all of which are asymptotic values of $f$ at infinity. Let $h=s_n\circ f$. Then by Lemma \ref{lem: item comparison 1}, $h$ has a $(K,\widetilde{\Sigma})$-quasiregular value with respect to the spherical metric at each of the points $s_n(y_1),\cdots,s_n(y_q)\in\S^n$, all of which are asymptotic values of $h$ at infinity, where $\widetilde{\Sigma}=C(n,y_1,\cdots,y_q)\Sigma$. Now, Proposition \ref{prop:great_picard_for_asymptotic_values} yields an upper bound on $q$ dependent only on $n$ and $K$. This completes the proof.
\end{proof}

\section{The Riemann removability theorem at infinity}\label{sec: riemann remo sec}

In this section, we prove Theorem \ref{thm: riemann remo at infty}, a Riemann removability theorem at infinity for bounded mappings satisfying the distortion inequality with defect \eqref{ineq: heterogeneous}. As a consequence, we also manage to strengthen the asymptotic Liouville-type theorem which was conjectured by Astala, Iwaniec and Martin in \cite[Section 8.5]{astala2008elliptic}, and proven by Kangasniemi and Onninen in \cite{kangasniemi2022heterogeneous,kangasniemi2024correction}.

We first recall several integrability estimates for mappings satisfying \eqref{ineq: heterogeneous}. The ideas behind these results are well understood, though not all of them have been stated in the exact form given here.

\begin{lem}[{\cite[Lemma 5.1]{kangasniemi2022heterogeneous}}]
\label{lem: Df is L^n}
    Suppose that $f\in W^{1,n}_\loc(\R^n,\R^n)$ satisfies the heterogeneous distortion inequality \eqref{ineq: heterogeneous} with $K\ge 1$ and $\Sigma\in L^1(\R^n)$. If $f$ is bounded, then $|Df|\in L^n(\R^n)$.
\end{lem}

The preceding lemma shows that every bounded solution $f$ of \eqref{ineq: heterogeneous} satisfies $|Df|\in L^n(\R^n)$. Once this global $L^n$-integrability is available, the $L^{1+\varepsilon}$-integrability of $\Sigma$ yields higher integrability for $Df$.

\begin{lem}\label{lem: higher integrability}
    Suppose that $f\in W^{1,n}_\loc(\R^n,\R^n)$, that $f$ is continuous, and that $|Df|\in L^n(\R^n)$. If $f$ satisfies \eqref{ineq: heterogeneous} with $\Sigma\in L^{1+\varepsilon}(\R^n)\cap L^1(\R^n)$ for some $\varepsilon>0$, then there exists $\varepsilon'\in(0,\varepsilon)$ such that
    \begin{equation*}
        \int_{\R^n}|Df|^{(1+\varepsilon')n}\lesssim_{n,K}\int_{\R^n}\Sigma^{1+\varepsilon'}<\infty.
    \end{equation*}
\end{lem}
\begin{proof}
    See the proof of \cite[Lemma 4.5]{kangasniemi2024quasiregular}; there, the proof is given for a target of $\R \times \S^{n-1}$, but is also valid for a target of $\R^n$. See also the analogous statements for an $\R^n$-target in \cite[Lemma 7.1]{kangasniemi2025single} and \cite[Theorem 1.4]{dolevzalova2024mappings}, where only local higher integrability is concluded due to using a local version of Gehring's lemma instead of a global one.
\end{proof}

We then note that under the assumption $|Df|\in L^n(\R^n)$, the $L^{1-\varepsilon}$-integrability of $\Sigma$ yields a logarithmic integrability estimate below the exponent $n$, involving the Hardy-Littlewood maximal function.

\begin{lem}\label{lem: lower integrability}
    Suppose that $f\in  W^{1,n}_\loc(\R^n,\R^n)$, that $f$ is continuous, and that $|Df|\in L^n(\R^n)$. If $f$ satisfies \eqref{ineq: heterogeneous} with $\Sigma\in L^{1-\varepsilon}(\R^n)\cap L^1(\R^n)$ for some $\varepsilon>0$, then 
    \begin{equation*}
        \int_{\R^n}|Df|^n\log^n\left(1+\frac{1}{M(|Df|)}\right)<\infty,
    \end{equation*}
    where $M$ stands for the Hardy-Littlewood maximal function.
\end{lem}
\begin{proof}
    The proof is identical to the one of \cite[Lemma 7.2]{kangasniemi2024correction}. Notably, the statement in \cite{kangasniemi2024correction} is again given for a map $f \colon \R^n \to \R \times \S^{n-1}$ instead of a map $f \colon \R^n \to \R^n$, but none of the individual steps use any properties specific to $\S^{n-1}$ that are not also valid in $\R^{n-1}$.
\end{proof}

We now combine these estimates to establish a version of Theorem \ref{thm: riemann remo at infty} for entire functions. The idea is that we follow the strategy used in the proof of \cite[Proposition 4.3]{kangasniemi2024quasiregular}, except we set up the chains of balls used in the proof in a manner that they avoid an increasingly large ball at the center of the domain. With this refinement, we are able to improve the boundedness shown in the original proof to convergence at infinity.

\begin{prop}\label{prop: riemann remo on R^n}
    Suppose that $f\in W^{1,n}_\loc(\R^n,\R^n)$ satisfies the distortion inequality with defect \eqref{ineq: heterogeneous}
    with $K\ge 1$ and $\Sigma\in L^{1\pm\varepsilon}(\R^n)$ for some $\varepsilon>0$. If $f$ is bounded, then $f$ has a limit at infinity.
\end{prop}

\begin{proof}
    By Lemma \ref{lem: higher integrability} we have $|Df|\in L^{n+\varepsilon'}(\R^n)$ for some $\varepsilon'\in(0,\varepsilon)$. Moreover, by Lemma \ref{lem: Df is L^n} together with the Hardy-Littlewood maximal inequality, $M(|Df|)\in L^n(\R^n)$, and by Lemma \ref{lem: lower integrability}, $|Df|\log(1+M^{-1}(|Df|))\in L^n(\R^n)$. Hence, for every $\delta>0$, we may choose $R>0$ sufficiently large such that
    \begin{equation*}
        \int_{\R^n\setminus \B^n(0,R)}|Df|^{n+\varepsilon'}<\delta^{n+\eps},
        \quad
        \int_{\R^n\setminus \B^n(0,R)} M^n(|Df|)<\delta,
    \end{equation*}
    and
    \begin{equation*}
        \int_{\R^n\setminus \B^n(0,R)}|Df|^n\log^n\left(1+\frac{1}{M(|Df|)}\right)<\delta^n.
    \end{equation*}
    
    We now fix such an $R>0$. Let $u=f_i$ be an arbitrary component function of $f$. Our aim is to show that 
    \begin{equation}\label{ineq: aim in section 6}
        |u(x)-u(y)|\lesssim_{n,\varepsilon'}\delta
    \end{equation}
    whenever $|x|,|y|>2R$. Since this can be done for every component function $f_i$, it follows that $f$ has a limit at infinity. 
    
    Let \(x,y\in\mathbb R^n\setminus \B^n(0,2R)\), and set 
    \[
        \rho:=\max(|x|,|y|),
        \quad \bar{x}:=\rho x/|x|,\quad {\rm and}\quad \bar{y}:=\rho y/|y|.
    \] 
    Then \(\bar{x},\bar{y}\in \partial \B^n(0,\rho)\). Choose a finite chain of points $P_0,\ldots,P_N\in \partial \B^n(0,\rho)$ along a spherical arc joining $\bar{x}$ to $\bar{y}$, where $N=N(n)$ depends only on $n$, such that  $P_0=\bar{x}$, $P_N=\bar{y}$ and  $|P_i-P_{i+1}|\le\rho$ for $i=0,1,...,N-1$. 

    We next estimate $|u(P_i)-u(P_{i+1})|$. Fix \(i\in\{0,\ldots,N-1\}\), and set $d_i:=|P_i-P_{i+1}|$. We use a standard chain-of-balls argument; see for instance \cite{hajlasz2000sobolev}. For \(j\in\mathbb Z\), let
    $r_j=2^{-|j|-2}d_i$. We define the centers \(x_j\) by
    \[
        x_j:=P_i+2^{-|j|-1}(P_{i+1}-P_i),\quad j\le 0,
    \]
    and
    \[
    x_j:=P_{i+1}-2^{-|j|-1}(P_{i+1}-P_i),\quad j\ge 0.
    \]
    Then the balls \(B_j:=(x_j,r_j)\) form a chain from \(P_i\) to \(P_{i+1}\), and the center of each $B_j$ is on the boundary of $B_{j-\sgn(j)}$ for $j\neq0$. 
    
    Moreover, no point in $\R^n$ is contained in more than two balls $B_j$, and the overlap of consecutive balls $B_j\cap B_{j-\sgn(j)}$ contains a ball $B_j'$ with radius $r'_j=r_j/2$. We also observe that every 
    \(B_j\) is contained in
    \(\R^n\setminus \B^n(0,R)\). Indeed, since $d_i=|P_i-P_{i+1}|\le \rho$, the line segment joining \(P_i\) and 
    \(P_{i+1}\) has distance at least $\sqrt{3}\rho/2$
    from the origin. On the other hand, \(r_j\le d_i/4\le\rho/4\) for all $j$. Hence
    \[
        \dist(B_j,0)\ge
        \left(\frac{\sqrt3 \rho}{2}-\frac{\rho}{4}\right)\ge\frac{\rho}{2}\ge R.
    \]
    See Figure \ref{fig:chain_of_balls} for an illustration of the chain of balls.
    
    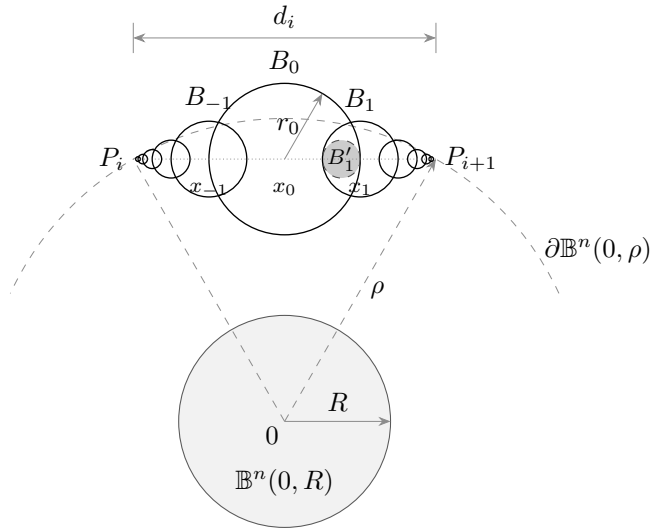
\begin{figure}[htbp]
    \centering
    \begin{tikzpicture}[scale=2, >=Stealth,
        every node/.style={font=\small,text=black},
        chain/.style={black,line width=0.5pt},
        auxiliary/.style={black!45,line width=0.35pt}]
        \coordinate (P) at (0,0);
        \coordinate (Q) at (2,0);
        \coordinate (O) at (1,{-sqrt(3)});

        % The excluded ball and the sphere containing P_i and P_{i+1}.
        \filldraw[fill=black!5,draw=black!65] (O) circle (0.7);
        \draw[auxiliary,dashed] (O)--(P);
        \draw[auxiliary,dashed,->] (O)--node[pos=0.5,right] {$\rho$} (Q) ;
        \draw[auxiliary,dashed]
        ({1+2*cos(25)},{-sqrt(3)+2*sin(25)})
        arc[start angle=25,end angle=155,radius=2];
        \node[anchor=west] at (2.65,-0.6) {$\partial \B^n(0,\rho)$};
        \draw[auxiliary,->] (O)--++(0.7,0)
        node[midway,above] {$R$};
        \fill (O) circle (0);
        \node[below left,inner sep=2pt] at (O) {$0$};
        \node at (1,{-sqrt(3)-0.40}) {$\B^n(0,R)$};
        \fill[black!20] (11/8,0) circle (1/8);

        % B_0 and the first five balls approaching each endpoint.
        \draw[chain] (1,0) circle (1/2);
        \foreach \j in {1,...,5} {
        \pgfmathsetmacro{\a}{2^(-\j)}
        \pgfmathsetmacro{\r}{\a/2}
        \draw[chain] (\a,0) circle (\r);
        \draw[chain] ({2-\a},0) circle (\r);
        }
        \draw[auxiliary,densely dotted] (P)--(Q);
        \draw[black!70,dashed,line width=0.4pt]
        (11/8,0) circle (1/8);

        % Centers, ball labels, and the radius of B_0.
        \foreach \x/\lab in {0.5/{-1},1/{0},1.5/{1}} {
        \fill (\x,0) circle (0.0);
        \node[below=5pt,font=\scriptsize] at (\x,0) {$x_{\lab}$};
        }
        \node[above] at (1,0.5) {$B_0$};
        \node[above] at (0.5,0.25) {$B_{-1}$};
        \node[above] at (1.5,0.25) {$B_1$};
        \draw[auxiliary,->] (1,0)--(5/4,{sqrt(3)/4})
        node[midway,left,inner sep=1pt] {$r_0$};
        \node[font=\scriptsize,inner sep=0pt] at (11/8,0) {$B'_1$};

    % Endpoints and their distance.
        \fill (P) circle (0);
        \fill (Q) circle (0);
        \node[anchor=east,inner sep=3pt] at (P) {$P_i$};
        \node[anchor=west,inner sep=3pt] at (Q) {$P_{i+1}$};
        \draw[auxiliary] (0,0.7)--(0,0.9) (2,0.7)--(2,0.9);
        \draw[auxiliary,<->] (0,0.8)--node[above] {$d_i$} (2,0.8);
        \end{tikzpicture}
        \caption{The chain of balls outside $\B^n(0,R)$}
        \label{fig:chain_of_balls}
    \end{figure}

    By continuity, we also have that the integral averages $u_{B_j}$ converge to $u(P_i)$ as $i\to-\infty$, and to $u(P_{i+1})$ as $j\to\infty$. We thus obtain an estimate
    \begin{equation*}
        |u(P_i)-u(P_{i+1})|\le \sum_{j=-\infty}^\infty|u_{B_{j+1}}-u_{B_j}|.
    \end{equation*}
    We prove the estimate only for the upper end of the series, corresponding to \(j\ge 0\), since the lower end \(j<0\) can be treated analogously. By using the balls $B'_{j+1}$ and by applying the Sobolev-Poincar\'e inequality, we obtain
    \begin{equation*}
        \begin{aligned}
        |u_{B_{j+1}}-u_{B_j}|&\le|u_{B_{j+1}}-u_{B'_{j+1}}|+|u_{B_{j}}-u_{B'_{j+1}}|\\
        &\le\dashint_{B'_{j+1}}|u-u_{B_j}|+\dashint_{B'_{j+1}}|u-u_{B_{j+1}}|\\
        &\lesssim_n\dashint_{B_j}|u-u_{B_j}|+\dashint_{B_{j+1}}|u-u_{B_{j+1}}|\\
        &\lesssim_nr_j\dashint_{B_j}|\nabla u|+r_{j+1}\dashint_{B_{j+1}}|\nabla u|.
        \end{aligned}
    \end{equation*}
    Since \(|\nabla u|\le |Df|\), it follows that
    \begin{equation*}
        \sum_{j=0}^\infty|u_{B_{j+1}}-u_{B_j}|\lesssim_n\sum_{j=0}^\infty r_j^{1-n}\int_{B_j}|Df|
    \end{equation*}
    Noting that \(r_j\) decreases to zero as \(j\to\infty\), 
    let \(j_0\in\mathbb Z_{\ge0}\)
    be the smallest index such that \(r_{j_0}\le2\).
    We first estimate the end of the series corresponding to \(j\ge j_0\). By H\"older's inequality, 
    \begin{multline*}
        \sum_{j=j_0}^\infty r_j^{-(n-1)}\int_{B_j}|Df|\le\sum_{j=j_0}^\infty r_j^{-(n-1)}(r_j^{n})^{\frac{n+\varepsilon'-1}{n+\varepsilon'}}\left(\int_{B_j}|Df|^{n+\varepsilon'}\right)^{\frac{1}{n+\varepsilon'}}\\
        =\sum_{j=j_0}^\infty r_j^{\frac{\varepsilon'}{n+\varepsilon'}}\left(\int_{B_j}| Df|^{n+\varepsilon'}\right)^{\frac{1}{n+\varepsilon'}}\lesssim_{n,\varepsilon'} r^{\frac{\varepsilon'}{n+\varepsilon'}}_{j_0}\|Df\|_{L^{n+\varepsilon'}(\R^n\setminus{\B^n(0,R)})}\lesssim_n\delta.
    \end{multline*}
    In particular, this upper bound is independent on $P_i$ and $P_{i+1}$.
    
    For the remaining part $0\le j<j_0$, we use the following elementary inequality: if
    $\Phi_1$, $\Phi_2$ are positive-valued real functions on an interval $I\subset\R$ with $\Phi_1$ increasing and $\Phi_2$ decreasing, then
    \begin{equation*}
        1\le\frac{\Phi_1(a)}{\Phi_1(b)}+\frac{\Phi_2(a)}{\Phi_2(b)}
    \end{equation*}
    for all $a,b\in I$. We apply this with 
    \begin{equation*}
        \begin{aligned}
            I=(0,\infty),&
            \quad\Phi_1(t)=t^{n-1},\quad\Phi_2(t)=\log(1+t^{-1}),\\
            &a=M(|Df|),
            \quad\text{and}\quad
            b=r_j^{-1/2}.
        \end{aligned}
    \end{equation*}
    Thus, for a.e. \(x\in B_j\) with \(M(|Df|)(x)>0\),
    we have
    \begin{equation*}
        1\le r_j^{\frac{n-1}{2}}M^{n-1}(|Df|)+\frac{\log(1+1/M(|Df|)}{\log(1+\sqrt{r_j})}.
    \end{equation*}
    Since $r_j>2$ for $0\le j<j_0$, 
    \[
        \log(1+\sqrt{r_j})>
        \log(\sqrt{r_j})=(\log r_j)/2>0.
    \]
    Hence,
    \begin{equation*}
        1\le r_j^{\frac{n-1}{2}}M^{n-1}(|Df|)+\frac{2}{\log r_j}\log\left(1+\frac{1}{M(|Df|)}\right).
    \end{equation*}
    Multiplying by \(r_j^{-(n-1)}|Df|\), integrating over 
    \(B_j\), and summing over $j=0,\cdots,j_0-1$, we obtain
    \begin{multline}\label{ineq: finite-ball part}
        \sum_{j=0}^{j_0-1}r_j^{1-n}\int_{B_j}|Df| \leq \sum_{j=0}^{j_0-1}r_j^{-\frac{n-1}{2}} \int_{B_j}|Df|M^{n-1}(|Df|) \\ 
        +\sum_{j=0}^{j_0-1}\frac{2}{r_j^{n-1}\log r_j} \int_{B_j} |Df| \log\left(1+\frac{1}{M(|Df|)}\right).
    \end{multline}
    Next, we estimate the first sum on the right hand side of \eqref{ineq: finite-ball part}. Since $|Df|\le M(|Df|)$ a.e.\ and by the choice of $j_0$ and $r_j$, we have $r_{j_0-1}>2$ and $r_j=2r_{j+1}$. Therefore, 
    \begin{multline*}
        \sum_{j=0}^{j_0-1}r_j^{-\frac{n-1}{2}}\int_{B_j}|Df|M^{n-1}(|Df|)\\
        \le \sum_{j=0}^{j_0-1}r_j^{-\frac{(n-1)}{2}}\left(\int_{\R^n\setminus \B^n(0,R)}M^n(|Df|)\right)\lesssim_n\delta,
    \end{multline*}
    which is again an upper bound in terms of $\delta$ that is independent on $P_i$ and $P_{i+1}$.
    
    For the other sum on the 
    the right hand side of \eqref{ineq: finite-ball part}, we use both the integral and sum versions of  H\"older's inequality, the fact that the balls \(B_j\) overlap with multiplicity at most \(2\), and the above estimate $r_j>2^{j_0-j}$. We obtain that
    \begin{align*}
        &\sum_{j=0}^{j_0-1}\frac{1}{r_j^{n-1}\log r_j}\int_{B_j}|Df|\log\left(1+\frac{1}{M(|Df|)}\right)\\
        &\quad\lesssim_n\sum_{j=0}^{j_0-1}\frac{1}{\log (r_j)}\left(\int_{B_j}|Df|^n\log^n\left(1+\frac{1}{M(|Df|)}\right)\right)^{\frac{1}{n}}\\
        &\quad\le \left(\sum_{j=0}^{j_0-1}\frac{1}{\log^{\frac{n}{n-1}} (r_j)}\right)^{\frac{n-1}{n}}
        \left(\sum_{j=0}^{j_0-1}\int_{B_j}|Df|^n\log^n\left(1+\frac{1}{M(|Df|)}\right)\right)^{\frac{1}{n}}\\
        &\quad\le\left(\sum_{j=1}^{j_0}\frac{1}{(j\log 2)^{\frac{n}{n-1}}}\right)^{\frac{n-1}{n}}\left(2\int_{\R^n\setminus \B^n(0,R)}|Df|^n\log^n\left(1+\frac{1}{M(|Df|)}\right)\right)^{\frac{1}{n}}\\
        &\quad\lesssim_n\delta.
    \end{align*}
    This upper bound is also independent of $P_i$ and $P_{i+1}$. Thus, combining our estimates, we obtain a bound for the $j\ge 0$ part of the telescoping sum of integral averages that is independent of the choice of $P_i$ and $P_{i+1}$. A similar estimate for the $j<0$ part yields
    \begin{equation*}
        |u(P_i)-u(P_{i+1})|\lesssim_{n,\varepsilon'}\delta
    \end{equation*}
    for $i=0,1,\cdots,N$. Consequently,
    \begin{equation}\label{ineq: u(barx)-u(bary)}
        |u(\bar x)-u(\bar y)|\le \sum_{i=0}^{N(n)}|u(P_i)-u(P_{i+1})|\lesssim_{n,\varepsilon'}\delta.
    \end{equation}
    It remains to estimate $|u(x)-u(\bar{x})|$ and $|u(y)-u(\bar{y})|$. We only consider the first term, since the second one is analogous. If $x=\bar{x}$, there is nothing to prove. Otherwise, we apply the same chain-of-balls construction to the radial line segment joining $x$ to $\bar{x}$. More precisely, let 
    \[ 
        d_x:=|\bar{x}-x|=\rho-|x|. 
    \] 
    For $j\in\mathbb{Z}$, define \[ 
        s_j:=2^{-|j|-2}d_x, 
    \] and choose centers on the radial segment $[x,\bar{x}]$ in the same manner as above. Since this segment is radial and every point on it has norm at least $|x|>2R$, the corresponding balls are contained in $\mathbb{R}^n\setminus \B^n(0,R)$. Therefore, the estimates above apply without modification, and we obtain 
    \begin{equation*}
        |u(x)-u(\bar{x})|
        \lesssim_{n,\varepsilon'}
        \delta. 
    \end{equation*}
    Similarly, $|u(y)-u(\bar{y})|
    \lesssim_{n,\varepsilon'}
    \delta.$
    Combining these estimates with \eqref{ineq: u(barx)-u(bary)}, we obtain 
    \[ 
        \begin{aligned} 
            |u(x)-u(y)| \leq |u(x)-u(\bar{x})| + |u(\bar{x})-u(\bar{y})| + |u(\bar{y})-u(y)|  \lesssim_{n,\varepsilon'}\delta
        \end{aligned} 
    \] 
    whenever $|x|,|y|>2R$. This completes the proof.
\end{proof}

We are now ready to prove Theorem~\ref{thm: riemann remo at infty}, an analogue of the Riemann removability theorem at infinity. The idea of the proof is simply to extend the map into the missing region and use Proposition~\ref{prop: riemann remo on R^n}. We first recall the statement for the convenience of the reader.

\begin{customthm}{\ref{thm: riemann remo at infty}}
    Suppose that $f\in W^{1,n}_\loc(\R^n\setminus \overline{\B^n},\R^n)$ satisfies the distortion inequality with defect \eqref{ineq: heterogeneous}
    with $K\ge 1$ and $\Sigma\in L^{1\pm\varepsilon}(\R^n\setminus \overline{\B^n})$ for some $\varepsilon>0$. If $f$ is bounded, then $f$ has a limit at infinity.
\end{customthm}

\begin{proof}
Choose a smooth function $\eta\in C^\infty(\R^n,[0,1])$ such that $\eta\equiv1$ on $\R^n\setminus 3\B^n$, $\spt \eta \cap \overline{2\B^n} = \emptyset$, and $|\nabla\eta|\le 2$. Define \(\widetilde f:\R^n\to\R^n\) by
\[
    \widetilde f(x)=
    \begin{cases}
        0, & |x|\le2,\\[1mm]
        \eta(x)f(x), & 2<|x|.
    \end{cases}
\]
It is immediate from the construction that $\widetilde{f}$ is continuous, bounded, and in $W^{1,p}_\loc(\R^n, \R^n)$ for some $p > n$.

We then show that there exists a non-negative $\widetilde\Sigma\in L^{1\pm\varepsilon'}(\R^n)$ for some $\eps' > 0$ such that 
\(\widetilde f\)
satisfies the distortion inequality with defect
\begin{equation}\label{eq:extended-heterogeneous}
    |D\widetilde f|^n
    \le K J_{\widetilde f}
    +\widetilde\Sigma
\end{equation}
a.e.\ in $\R^n$. The only set where this is non-trivial is the annulus $\A^n(2,3)$, as we may set $\widetilde{\Sigma} = \Sigma$ in $\R^n \setminus 3\B^n$ and $\widetilde{\Sigma} \equiv 0$ in $2\B^n$. On the annulus $\A^n(2,3)$, we have
\[
    D\widetilde f
    =\eta Df+f\otimes \nabla\eta,
\]
and hence
\begin{equation}\label{eq:Dftilde-cutoff}
    |D\widetilde f|^n
    \le C_1
    \left(\eta^n|Df|^n+|f|^n|
    \nabla\eta|^n\right),
\end{equation}
for some constant $C_1=C_1(n)$. Moreover, by standard algebra of differential forms, 
\[
    J_{\widetilde{f}} \vol_n
    = \eta^n J_f \vol_n + \sum_{i=1}^n n \eta^{n-1} f_i df_1 \wedge \dots \wedge df_{i-1} \wedge d\eta \wedge df_{i+1} + \dots \wedge df_n,
\]
and hence
\[
    \eta^nJ_f
    \le
    J_{\widetilde f}
    +C_2\eta^{n-1}|Df|^{n-1}
    |f||\nabla\eta|.
\]
for some constant $C_2=C_2(n)$. Consequently,
\begin{multline}\label{ineq: eta-Df-1}
    \eta^n|Df|^n\le
    \eta^nKJ_f+
    \eta^n\Sigma\\
    \le
    KJ_{\widetilde f}
    +KC_2\eta^{n-1}|Df|^{n-1}
        |f||\nabla\eta|
    +\eta^n\Sigma.
\end{multline}
Applying Young's inequality to the second term on the right hand side of \eqref{ineq: eta-Df-1}, we obtain
\[
    KC_2\eta^{n-1}|Df|^{n-1}
        |f||\nabla\eta|
    \le
    \frac12\eta^n|Df|^n
    +C_3|f|^n|\nabla\eta|^n.
\]
for some constant $C_3=C_3(n,K)$.
Absorbing the term $\eta^n|Df|^n$ into the left- hand side of \eqref{ineq: eta-Df-1}, we have 
\begin{equation}\label{ineq: eta-Df-2}
    \eta^n|Df|^n
    \le
    2KJ_{\widetilde f}
    +2\eta^n\Sigma+2 C_3|f|^n|\nabla\eta|^n.
\end{equation}
Combining \eqref{eq:Dftilde-cutoff} and
\eqref{ineq: eta-Df-2}, and using $\eta\le 1$ and $|\nabla \eta|\le2$, we obtain
\[
    |D\widetilde f|^n
    \le (2C_1 K)J_{\widetilde{f}}
    +2C_1\Sigma+ 2^nC_1(2C_3+1)|f|^n
\]
for a.e. $x\in\A^n(2,3)$. Thus, by setting setting 
\[
    \widetilde\Sigma = \max(0, (2C_1 K - K)J_{\widetilde{f}} + 2C_1\Sigma + 2^nC_1(2C_3+1)|f|^n)
\] 
in $\A^n(2,3)$, the proof of \eqref{eq:extended-heterogeneous} is complete.

Since \(f\) is bounded,
\(\Sigma\in L^{1\pm\varepsilon}(\R^n\setminus\overline{\B^n})\), and $\widetilde f \in W^{1,p}_\loc(\R^n, \R^n)$ with $p > n$, it follows that $\widetilde\Sigma\in L^{1\pm\varepsilon'}(\R^n)$ for some $\eps' > 0$.
We may therefore apply Theorem
\ref{thm: riemann remo at infty} to 
\(\widetilde f\) and conclude that \(\widetilde f\) has a limit at infinity. Since 
\(\widetilde f=f\) on
\(\R^n\setminus3\B^n\), it follows that \(f\) also has a limit at infinity. This completes the proof.
\end{proof}

As an immediate consequence of Theorem~\ref{thm: riemann remo at infty}, we obtain the following strengthening of the Liouville-type theorem in \cite{kangasniemi2022heterogeneous, kangasniemi2024correction}. We recall the statement for the convenience of the reader.

\begin{customthm}{\ref{thm: liouville at omitted point}}
    Let $K\ge 1$ and $\Sigma\in L^{1\pm\varepsilon}(\R^n)$ for some $\varepsilon>0$. Suppose that $f\in W^{1,n}_\loc(\R^n,\R^n)$ has a $(K,\Sigma)$-quasiregular value at $y_0\in \R^n$. If $f$ is bounded and $y_0$ is an asymptotic value of $f$ at infinity, then $f \equiv y_0$.
\end{customthm}

\begin{proof}
    Since $f$ is bounded and $\Sigma\in L^{1\pm\varepsilon}(\R^n)$, the function $\widetilde{\Sigma}
    :=\Sigma|f-y_0|^n\in L^{1\pm\varepsilon}(\R^n)$ and we have 
    \[
        |Df(x)|^n\le KJ_f(x)+\widetilde{\Sigma}(x),
    \]
    for a.e. $x\in\R^n$. 
    By Theorem \ref{thm: riemann remo at infty}, such a mapping $f$ has a limit at infinity. Since $y_0$ is an asymptotic value of $f$ at infinity, this limit must equal $y_0$. The conclusion then follows directly from \cite[Theorem 1.3]{kangasniemi2024correction}. 
\end{proof}

\section{Proofs of the main theorems}\label{sec: classical great picard}

We are now ready to assemble together the pieces we have shown into a complete proof of our main results, Theorems~\ref{thm: classical great picard thm} and \ref{thm: classical great picard thm bounded}. We first recall the statement of Theorem~\ref{thm: classical great picard thm}.

\begin{customthm}{\ref{thm: classical great picard thm}}
    Let $K\ge 1$ and $\Sigma\in L^{1\pm\varepsilon}(\R^n\setminus\overline{\B^n})$ for some $\varepsilon>0$. If a continuous map $f\in W^{1,n}_{\loc}(\R^n \setminus \overline{\B^n},\R^n)$ omits  $q=q(n,K)\in\Z_{>0}$ distinct $(K,\Sigma)$-quasiregular values $y_1,...,y_q\in \R^n\setminus f(\R^n\setminus\overline{\B^n})$, then the limit $\lim_{|x|\to\infty}
    f(x)\in\overline{\R^n}$ exists.
\end{customthm}

\begin{proof}
    Let $q=q(n,K)$ be the positive integer given by
    Corollary~\ref{cor: great picard type thm into R^n}.
    Suppose that $f$ has $(K,\Sigma)$-quasiregular values at the
    $q$ distinct omitted points $y_1,\ldots,y_q$. Then by Corollary~\ref{cor: great picard type thm into R^n}, at least one of the points
    $y_1,\ldots,y_q$ is not an asymptotic value of $f$ at infinity. We may therefore assume that $y_1$ is not an
    asymptotic value of $f$ at infinity.
    Then there exist $R>1$ and $\delta>0$ such that
    \begin{equation}\label{eq: away from y1 near infinity}
        |f(x)-y_1|\ge \delta
        \qquad
        \text{for every } |x|\ge R.
    \end{equation}

    Fix $R>1$ and $\delta>0$ as in
    \eqref{eq: away from y1 near infinity}, and define
    \[
        F:\R^n\setminus\overline{\B^n}\to\R^n,\qquad
        F(x):=f(Rx).
    \]
    Then $F$ is continuous and belongs to
    $W^{1,n}_{\loc}\bigl(\R^n\setminus
    \overline{\B^n},\R^n\bigr)$. Moreover, by a standard application of the chain rule, $F$ has a $(K,\Sigma_R)$-quasiregular value at $y_1$,
    where $\Sigma_R \in L^{1\pm\varepsilon}(\R^n\setminus\overline{\B^n})$ is given by
    \[
        \Sigma_R(x):=R^n\Sigma(Rx).
    \]

    Let $\mathcal R\in O(n)$ be a reflection and define
    \[
        \iota: \R^n\setminus\{y_1\}\to\R^n\setminus\{0\},\quad
        \iota(y)
        :=
        \frac{\mathcal R(x-y_1)}{|x-y_1|^2}.
    \]
    Then map $\iota$ is an orientation-preserving conformal
    diffeomorphism satisfying
    \[
        |D\iota(x)|^n=
        J_\iota(x)=|x-y_1|^{-2n}.
    \] 
    Set $g:=\iota\circ F$. By \eqref{eq: away from y1 near infinity}, we have 
    \begin{equation}\label{eq: F separated}
        |F(x)-y_1|\ge\delta
        \qquad
        \text{for every } |x|>1.
    \end{equation}
    Thus, $y_1$ is omitted by $F$, and the map $g$ is well defined,
    continuous, and belongs to $W^{1,n}_{\loc}\bigl(
    \R^n\setminus\overline{\B^n},\mathbb{R}^n\bigr)$. Moreover, by \eqref{eq: F separated},
    \[
        |g(x)|=|F(x)-y_1|^{-1}\le \delta^{-1}
    \]
    is bounded on $\R^n\setminus\overline{\B^n}$, and a standard computation using the chain rule yields that, 
    \begin{multline*}
        |Dg(x)|^n = \abs{F(x)-y_1}^{-2n} \abs{DF(x)}^n\\
        \le \abs{F(x)-y_1}^{-2n} \abs{DF(x)}^n
            + \abs{F(x)-y_1}^{-n} \Sigma_R(x)\\
        \le KJ_g(x)+\Sigma_R(x)|g(x)|^n
    \end{multline*}
    for a.e. $x\in\R^n\setminus
    \overline{\B^n}$. Thus, $g$ has a $(K,\Sigma_R)$-quasiregular value at $0$. Since $g$ is bounded, Theorem~\ref{thm: liouville at omitted point} yields that the limit $b:=\lim_{x\to\infty}g(x)$ exists in $\R^n$. It follows that
    \[
        \lim_{x\to\infty} f(x)=\lim_{x\to\infty} F(x)=
        \iota^{-1}(b),
    \]
    where we interpret that $\iota^{-1}(b) = \infty$ if $b = 0$.
    Hence, $f$ has a limit at infinity in $\overline{\R^n}$, completing the proof.
\end{proof}

\begin{rem}\label{rem:little_picard_from_great_picard}
    We point out how the Little Picard Theorem \ref{thm:Kangasniemi-Onninen_little_picard} for quasiregular values follows from Theorem \ref{thm: classical great picard thm}. Indeed, let $q=q(n,K) \ge 2$ be the constant of Theorem \ref{thm: classical great picard thm}, and suppose that $f \colon \R^n \to \R^n$ has $q$ distinct $(K, \Sigma)$-quasiregular values $\omega_1, \dots, \omega_q$ in $\partial f(\R^n)$, where $K \ge 1$ and $\Sigma \in L^{1\pm \eps}(\R^n)$. By Theorem \ref{thm: classical great picard thm}, $f$ has a limit $y_0 \in \overline{\R^n}$ at infinity. 
    
    But now, since every $\omega_i \in \partial f(\R^n)$, there exist sequences of points $x_{i,j} \in \R^n$ such that $\lim_{j \to \infty} f(x_{i,j}) = \omega_i$. Since $f$ is continuous, and since the openness part of the single-value Reshetnyak's theorem \cite[Theorem 1.2]{kangasniemi2025single} implies that $\omega_i \notin f(\R^n)$, we must hence have $\lim_{j \to \infty} x_{i,j} = \infty$ to prevent $x_{i,j}$ from having a convergent subsequence in $\R^n$. But now we have $\omega_i = y_0$. As this is true for every $i$, we obtain a contradiction with the points $\omega_i$ being distinct, showing that such an $f$ cannot exist. 
\end{rem}

Theorem \ref{thm: classical great picard thm bounded} is then a straightforward corollary of Theorem \ref{thm: classical great picard thm}. We recall the statement and give a proof.

\begin{customthm}{\ref{thm: classical great picard thm bounded}}
    Let $K\ge 1$ and $\Sigma\in L^{1+\varepsilon}(\B^n)$ for some $\varepsilon>0$. If a continuous map $f\in W^{1,n}_{\loc}(\B^n \setminus \{0\},\R^n)$ omits $q=q(n,K)\in\Z_{>0}$ distinct $(K,\Sigma)$-quasiregular values $y_1,...,y_q\in \R^n\setminus f(\B^n \setminus \{0\})$, then the limit $\lim_{|x|\to0}
    f(x)\in\overline{\R^n}$ exists.
\end{customthm}

\begin{proof}
Let $q=q(n,K)$ be the positive integer given by
Theorem~\ref{thm: classical great picard thm}. Suppose that
$f\in W^{1,n}_{\loc}(\B^n\setminus\{0\},\R^n)$ has
$(K,\Sigma)$-quasiregular values at $q$ distinct omitted points
\[
    y_1,\ldots,y_q\in\R^n\setminus f(\B^n\setminus\{0\}).
\]

Let $\mathcal R\in O(n)$ be a reflection and define the
orientation-preserving conformal inversion
\[
    \iota:\R^n\setminus\{0\}\to\R^n\setminus\{0\},
    \qquad
    \iota(x):=\frac{\mathcal R x}{|x|^2}.
\]
Then the restriction of $\iota$ maps
$\R^n\setminus\overline{\B^n}$ onto $\B^n\setminus\{0\}$, and
\[
    |D\iota(x)|^n=
    J_\iota(x)=|x|^{-2n}.
\]
Define $\widetilde{f}:=f\circ\iota$ on $\R^n\setminus\overline{\B^n}$ and set
\[
    \widetilde{\Sigma}(x)
    :=
    |x|^{-2n}\Sigma(\iota(x)).
\]
Then $\widetilde{f}\in W^{1,n}_{\loc}
(\R^n\setminus\overline{\B^n},\R^n)$, and the chain rule yields that $\widetilde{f}$ has $(K,\widetilde{\Sigma})$-quasiregular values at $y_1,\cdots,y_q$. Moreover,
\(
    F(\R^n\setminus\overline{\B^n})
    =
    f(\B^n\setminus\{0\}),
\)
so the points $y_1,\ldots,y_q$ remain omitted by $\widetilde{f}$.

We next verify the required integrability of
$\widetilde{\Sigma}$. Choose
\[
    0<\delta<\frac{\varepsilon}{1+2\varepsilon}.
\]
For every $p>0$, a change of variables under $\iota$ yields
\begin{multline*}     
    \int_{\R^n\setminus
    \overline{\B^n}}
    \widetilde{\Sigma}(x)^p\,dx
    =\int_{\R^n\setminus
    \overline{\B^n}}\Sigma^p(\iota(x))|x|^{-2n(p-1)}J_\iota(x)\, dx\\
    =\int_{\B^n\setminus\{0\}}
        \Sigma^p(y) |y|^{2n(p-1)}\,dy.
\end{multline*}
Taking $p=1+\delta$, and using $\delta<\varepsilon$, we obtain
\[
    \int_{\R^n\setminus
    \overline{\B^n}}
    \widetilde{\Sigma}^{1+\delta}(x)\,dx
    =\int_{\B^n\setminus\{0\}}
    \Sigma^{1+\delta}|y|^{2n\delta}
    \,dy \le
    \int_{\B^n\setminus\{0\}}\Sigma^{1+\delta}\,dy
    <\infty.
\]
On the other hand, taking $p=1-\delta$ and applying
H\"older's inequality yields
\begin{multline*}
    \int_{\R^n\setminus
    \overline{\B^n}}
    \widetilde{\Sigma}^{1-\delta}(x)\, dx
    =\int_{\B^n\setminus\{0\}}
    \Sigma^{1-\delta}|y|^{-2n\delta}
    \,dy \\
    \le
    \left(
        \int_{\B^n\setminus\{0\}}\Sigma^{1+\varepsilon}\,dy
    \right)^{\frac{1-\delta}{1+\varepsilon}}
    \left(
        \int_{\B^n\setminus\{0\}}
        |y|^{-\frac{2n\delta(1+\varepsilon)}
        {\varepsilon+\delta}}\,dy
    \right)^{\frac{\varepsilon+\delta}{1+\varepsilon}}.
\end{multline*}  
The last integral is finite, since the choice of $\delta$ ensures that
\[
    \frac{2n\delta(1+\varepsilon)}
         {\varepsilon+\delta}<n.
\]
Consequently, $\widetilde{\Sigma}\in
L^{1\pm\delta}(\R^n\setminus\overline{\B^n})$.
We may therefore apply
Theorem~\ref{thm: classical great picard thm} to $\widetilde{f}$ and conclude that $\widetilde{f}$ has a limit at infinity in $\overline{\R^n}$.
Since $\iota(x)\to\infty$ as $x\to0$, 
it follows that $f$ has a limit at zero in $\overline{\R^n}$. This completes the proof.
\end{proof}

We finish this paper by pointing out examples which show that Theorems~\ref{thm: classical great picard thm} and \ref{thm: classical great picard thm bounded} are sharp. 

For Theorem~\ref{thm: classical great picard thm}, we obtain that $\Sigma\in L^{1\pm\varepsilon}(\R^n \setminus \overline{\B^n})$ cannot be replaced by 
$(L^{1+\varepsilon}_{\loc}\cap L^1\cap L^{1-\varepsilon})(\R^n \setminus \overline{\B^n})$ or $(L^{1+\varepsilon}\cap L^1)(\R^n \setminus \overline{\B^n})$, by using the analogous counterexamples for the Little Picard theorem of quasiregular values given in \cite[Examples~9.1 and 9.2]{kangasniemi2024quasiregular}. Notably, the counterexamples for the Little Picard are based on constructing spikes in balls contained in $\R^n$. For the counterexamples to remain valid in our setting, we only need to select the balls used in the construction so that they do not meet $2\B^n$ and their centerpoints tend to infinity.

Finally, for Theorem \ref{thm: classical great picard thm bounded}, we obtain a counterexample showing that $\Sigma \in (L^1 \cap L^{1+\eps}_\loc)(\B^n \setminus \{0\})$ is insufficient by pre-composing the map given in \cite[Examples~9.1]{kangasniemi2024quasiregular} with a conformal inversion map. We leave the details to the interested reader.

\subsection*{Acknowledgments}

We thank Pekka Pankka for several helpful comments on the manuscript.

\subsection*{Statement on AI use}

Each author declares that AI models have not been used in their portion of the research and writing work of this article.

\bibliographystyle{abbrv}
\bibliography{Greatpicard.bib}
\end{document}